\documentclass{amsart}
\usepackage{tikz}

\usepackage{amsmath}
\usepackage{amssymb}
\usepackage{mathrsfs}
\usepackage[all]{xy}

\newtheorem{theorem}{Theorem}[section]
\newtheorem{claim}[theorem]{Claim}

\newtheorem{lemma}[theorem]{Lemma}

\newtheorem{proposition}[theorem]{Proposition}
\newtheorem{corollary}[theorem]{Corollary}

\newtheorem{fact}[theorem]{Fact}
\newtheorem*{thm.2.4}{Theorem \ref
{GlavinMainTheorem}}
\newtheorem*{cor2.23}{Corollary \ref{compactQpoint}}
\newtheorem*{thm.316}{Theorem \ref{thm.3.16}}
\newtheorem*{thm.317}{Theorem \ref{thm.3.17}}
\newtheorem*{cor.5.6}{Theorem \ref{ProductCor}}
\newtheorem*{cor.5.7}{Theorem \ref{PowerCor}}
\newtheorem*{thm6.10}{Theorem \ref{thm.bgclosedclass}}
\newtheorem*{thm.7.2}{Theorem
 \ref{analyse}}
\newtheorem*{thm.7.9}{Theorem \ref{Kunen-Paris-Model}}
\newtheorem*{thm.UA}{Theorem \ref{UACor}} 

\theoremstyle{definition}
\newtheorem{definition}[theorem]{Definition}
\newtheorem{example}[theorem]{Example}

\newtheorem{question}[theorem]{Question}

\theoremstyle{remark}
\newtheorem{remark}[theorem]{Remark}

\def\l{{\langle}}
\def\r{{\rangle}}

\newcount\skewfactor
\def\mathunderaccent#1#2 {\let\theaccent#1\skewfactor#2
\mathpalette\putaccentunder}
\def\putaccentunder#1#2{\oalign{$#1#2$\crcr\hidewidth
\vbox to.2ex{\hbox{$#1\skew\skewfactor\theaccent{}$}\vss}\hidewidth}}

\def\smallbox#1{\leavevmode\thinspace\hbox{\vrule\vtop{\vbox
   {\hrule\kern1pt\hbox{\vphantom{\tt/}\thinspace{\tt#1}\thinspace}}
   \kern1pt\hrule}\vrule}\thinspace}

\DeclareMathOperator{\dom}{dom}

\DeclareMathOperator{\Cub}{Cub}

\DeclareMathOperator{\Add}{Cohen}

\DeclareMathOperator{\otp}{otp}
\DeclareMathOperator{\id}{id}
\DeclareMathOperator{\Ch}{Ch}
\DeclareMathOperator{\Sp}{Sp}
\DeclareMathOperator{\Image}{Im}

\newcommand{\cf}{{\rm cf}}

\newcommand{\ra}{\rightarrow}
\newcommand{\al}{\alpha}
\newcommand{\om}{\omega}
\newcommand{\sse}{\subseteq}
\newcommand{\re}{\upharpoonright}
\DeclareMathOperator{\Gal}{Gal}

\newcommand{\contains}{\supseteq}

\title[Cofinal types over measurable cardinals]{Cofinal types of ultrafilters over\\ measurable cardinals}
\date{\today}

\author{Tom Benhamou}
\address[Benhamou]{Department of Mathematics, Statistics, and Computer Science, University of Illinois at Chicago, Chicago, IL 60607, USA}
\email{tomb@uic.edu}
\thanks{The research of the first author was supported by the National Science Foundation under Grant
DMS-2246703.}

\author{Natasha Dobrinen}

\address[Dobrinen]{Department of Mathematics, University of Notre Dame, Notre Dame, IN 46556, USA}
\email{ndobrine@nd.edu}
\thanks{The second author was partially supported by  National Science Foundation Grant DMS-1901753}


\begin{document}
\let\labeloriginal\label
\let\reforiginal\ref
\def\ref#1{\reforiginal{#1}}
\def\label#1{\labeloriginal{#1}}
\maketitle
\begin{abstract}
    We develop the theory of cofinal types of ultrafilters 
    over measurable cardinals and establish its connections to  Galvin's property. 
   We generalize fundamental results from the countable to the uncountable, 
  but   often in  surprisingly strengthened forms, 
    and present models with varying  structures of the cofinal types of ultrafilters over measurable cardinals.
\end{abstract}
\section{Introduction}

This article develops the theory of cofinal types of ultrafilters over  measurable cardinals. 
While 
cofinal, or Tukey, types of ultrafilters   over $\om$  have  been studied extensively (see, for instance, \cite{Blass/Dobrinen/Raghavan15,DobrinenJSL15,DobrinenFund20,Dobrinen/Mijares/Trujillo14,Dobrinen/Todorcevic11,Dobrinen/Todorcevic14,Dobrinen/Todorcevic15,Kuzeljevic/Raghavan18,Milovich08,Milovich12,Milovich20,Raghavan/Shelah17,Raghavan/Todorcevic12,Raghavan/Verner19}),
this is the first substantive investigation  of  cofinal types of    ultrafilters  
 over measurable cardinals.

Given two posets $(P,\leq_P)$ and $(Q,\leq_Q)$, a function $f:Q\rightarrow P$ is \textit{cofinal} if for every cofinal subset $X\subseteq Q$, its $f$-image $f''X$ is cofinal in $P$. A function $g:P\rightarrow Q$ is \textit{unbounded}
if for every unbounded  subset $X\subseteq P$,
$g''X$ is unbounded in $Q$.
Schmidt showed in  \cite{Schmidt55}
that there is a cofinal map  from $Q$ to $P$ if and only if there is an unbounded map 
from $P$ to $Q$.
We say that 
$P$ is {\em Tukey reducible to $Q$}, and write
$P\leq_T Q$, if there is a cofinal map  $f:Q\rightarrow P$  or, equivalently,  an unbounded map $g:P\rightarrow Q$.
When both $P\le_T Q$ and $Q\le_T P$, then we say that they are 
{\em Tukey equivalent},  and write $P\equiv_T Q$.
This forms an equivalence relation, and the equivalence classes are called {\em Tukey types}.
We regard ultrafilters  as directed partial orders under reverse inclusion. A cofinal subset of an ultrafilter $U$ with respect to reversed inclusion is also called a \textit{filter base}.

The Tukey order was introduced by Tukey in  \cite{Tukey40}  to extend  the notion of  Moore-Smith convergence in topology to the more general setting of nets, or  directed partial orders.
Two directed partial orders are {\em cofinally similar} exactly when there is a third directed partial order into which they each embed as a cofinal subset. 
In particular, for  ultrafilters Tukey equivalence is the same as cofinal similarity; hence  Tukey types are also called  {\em cofinal types}.

Tukey reduction has been seen to provide insightful classifications of partially ordered sets in contexts where isomorphism is too strong a notion to produce any meaningful results.  This has been  seen in work of Isbell \cite{Isbell65}, Todorcevic \cite{TodorcevicDirSets85, Todorcevic96}, Fremlin \cite{Fremlin91}, and Solecki--Todorcevic \cite{Solecki/Todorcevic04}, to mention a few, as well as  a large collection  papers on ultrafilters over countable base sets
beginning  with Milovich's paper \cite{Milovich08}.
The majority of results
 on cofinal types of ultrafilters over $\om$  fall into one of the following categories:
\begin{enumerate}
    \item 
Finding conditions under which  Fubini products are cofinally equivalent to  cartesian products \cite{Dobrinen/Todorcevic11};
\item
Canonizations of cofinal maps into continuous  or at least finitary maps which imply cofinal types of certain ultrafilters have  cardinality $\mathfrak{c}$ (rather than $2^{\mathfrak{c}}$)  \cite{Dobrinen/Todorcevic11, Raghavan/Todorcevic12, Blass/Dobrinen/Raghavan15, DobrinenJSL15, Dobrinen/Mijares/Trujillo14, DobrinenFund20}; 
\item
Other conditions guaranteeing an ultrafilter is not of the maximum cofinal type \cite{Milovich08, Dobrinen/Todorcevic11, Raghavan/Todorcevic12, Blass/Dobrinen/Raghavan15, DobrinenJSL15};
\item
Embeddings of certain partial orders  \cite{Dobrinen/Todorcevic11,Raghavan/Todorcevic12},
including long  lines \cite{Raghavan/Verner19, Kuzeljevic/Raghavan18} and $\mathcal{P}(\om)/\mathrm{fin}$ \cite{Raghavan/Shelah17}, into the cofinal types of p-points in various models of \textsf{ZFC};
\item
Finding conditions under which  Tukey reduction implies Rudin-Keisler or even  Rudin-Blass reduction \cite{Raghavan/Todorcevic12, DobrinenFund20};
\item
Finding the exact structure of the cofinal types, including the precise structure of the Rudin-Keisler classes inside them, below Ramsey ultrafilters
\cite{Raghavan/Todorcevic12},
below certain p-points satisfying weak partition relations \cite{Dobrinen/Todorcevic14, Dobrinen/Todorcevic15,Dobrinen/Mijares/Trujillo14}, 
and below non-p-points  forced by $\mathcal{P}(\om^k)/\mathrm{fin}^{\otimes k}$ \cite{DobrinenJSL15}.
\end{enumerate}
We extend some results from each of the categories (1)--(4) to 
 ultrafilters on measurable cardinals.
In contrast to the situation on $\om$, we will pinpoint the exact structure of the cofinal types in certain models with measurable cardinals, establishing results on $\om$ lying in category (6) but not exact analogues of them. Our results 
 will be discussed in Subsection \ref{Subsec:IntroResults}.

The present investigation began with a realization of the deep connection between the Galvin property (Definition \ref{Def:Galvin Property}) and 
the non-maximality of cofinal types of ultrafilters (Theorem \ref
{GlavinMainTheorem}).
The Galvin property for ultrafilters on measurable cardinals has been studied extensively  in recent years, and 
a major motivation for  our 
work  was to develop a finer understanding of  conditions guaranteeing 
 the Galvin property through the development of the Tukey theory of ultrafilters on measurable cardinals.

\subsection{The Galvin property}

In 1973, F.~Galvin proved the following surprising result \cite{MR795592}:
\begin{theorem}[Galvin's Theorem]
    Suppose that $\kappa^{<\kappa}=\kappa>\aleph_0$. Then for every normal filter\footnote{$U$ is normal if $\triangle_{\alpha<\kappa}F_\alpha:=\{\beta<\kappa\mid \forall\alpha<\beta\, (\beta\in F_\alpha)\}\in U$ for all $\langle F_\alpha\mid \alpha<\kappa\rangle\subseteq U.$} $U$ over $\kappa$, the Galvin property holds: for every sequence of sets $\l A_i\mid i<\kappa^+\r\in [U]^{\kappa^+}$, there is an $I\in [\kappa^+]^{\kappa}$ such that $\bigcap_{i\in I}A_i\in U$. 
\end{theorem}
For example, if $\kappa=\omega_1$ 
 and \textsf{GCH} holds
(so $\omega_1^{<\omega_1}=\omega_1$),
then the club filter $\Cub_{\omega_1}$ 
is normal and Galvin's Theorem asserts that for every collection of $\omega_2$ many clubs there is a subcollection of size $\omega_1$ whose intersection is still a club.

Although Abraham and Shelah \cite{AbrahamShelah1986} and later Benhamou, Garti, and Poveda \cite{bgp} show that the assumption that $\kappa^{<\kappa}=\kappa$ cannot be dropped from Galvin's Theorem, the normality assumption and the property obtained in Galvin's theorem were recently improved by Benhamou and Gitik \cite{Parttwo,SatInCan}, and by Garti \cite{MR3604115,MR3787522}. In this paper we will provide certain strengthenings of those results. 

A filter satisfying the property  described in Galvin's theorem is called a \textit{Galvin filter}. 
Galvin filters turned out to be useful for several applications described  below, which brought the Galvin property back to the front of the stage.
Let us use  notation from \cite{bgp} to allow some flexibility with certain parameters in Galvin's property: 
\begin{definition}[The Galvin property]\label{Def:Galvin Property}
    For an ultrafilter $U$ over $\kappa$ 
    and two cardinals $\mu\leq \lambda$ we denote by $\Gal(U,\mu,\lambda)$ the assertion that
    $$\forall \l A_i\mid i<\lambda\r\in[ U]^\lambda \ \exists I\in [\lambda]^\mu \ \cap_{i\in I}A_i\in U$$
A {\em Galvin ultrafilter} 
is an ultrafilter satisfying $\Gal(U,\kappa,2^\kappa)$.
\end{definition}

Before diving into deeper set-theoretic water, let us elaborate on some  applications of Galvin's property.

The use of filters and ultrafilters to establish theorems both in finite and infinitary combinatorics has a long history. Paradigmatic examples of this are   the \emph{Pigeonhole Principle}, 
\emph{Ramsey's Theorem}, \emph{Arrow's Impossibility Theorem}, or 
the partition relation $\kappa\rightarrow (\kappa)^2_2$ for measurable cardinals, 
among many others. See \cite{KomjathBook} for additional interesting examples.
It turns out that Galvin's property characterizes some \emph{polarized partition relations}. 
 Hajnal \cite{Haj} proved that if $\kappa$ is a measurable cardinal  then $\binom{\kappa^+}{\kappa}\rightarrow\binom{\kappa}{\kappa}_\theta$ for every $\theta<\kappa$.\footnote{This stands for the following statement: For every coloring $c\colon \kappa^+\times \kappa\rightarrow\theta$ there is $A\in [\kappa^+]^\kappa$ and $B\in [\kappa]^\kappa$ such that $c\restriction (A\times B)$ is constant.}
A modern proof of Hajnal's result using $\mathrm{Gal}(U,\kappa,\kappa^+)$ can be found in \cite[Proposition 3.3]{bgs}. This latter proof actually yields the enhanced polarized relation saying that the second component of the monochromatic set is a member of $U$ (i.e., $\binom{\kappa^+}{\kappa}\rightarrow\binom{\kappa}U_\theta$). 
Moreover, that  argument is   more flexible than other stronger versions of Galvin's property, such as the ones considered in \cite{bgp}, entailing stronger forms of the polarized partition relation. 
But most interestingly, the above implication is reversible, namely, $\binom{\kappa^+}{\kappa}\rightarrow\binom{\kappa}{U}_2$ implies  $\Gal(U,\kappa,\kappa^+)$\footnote{Given a sequence $\l A_\alpha\mid \alpha<\kappa\r$, the desired intersection of size $\kappa$ is given by the homogeneous rectangle guaranteed  by the assumption $\binom{\kappa^+}{\kappa}\rightarrow\binom{\kappa}{U}_2$, for the coloring $c:\kappa^+\times\kappa\rightarrow 2$ defined by $c(\alpha,\beta)=1$ iff $\beta\in A_\alpha$.}. 

Other interesting connections between Galvin's property and infinite combinatorics arise in the  context of \emph{prediction principles}; especially in the study of the \emph{diamond-like} ones (such as \emph{Devlin-Shelah weak diamond} in \cite{MR3604115,MR3914938} and \emph{the superclub principle} in \cite{MR3787522}, see also \cite{ForcingGalvin}) More recently, the Galvin property has been used by Benhamou, Garti and Poveda \cite{GalDet} to prove consistently new instances of the partition relation $\lambda\rightarrow (\lambda,\omega+1)$ and the classical question of \cite{MR795592}  about possible generalizations (both in \textsf{ZFC} and \textsf{ZF}+\textsf{AD}) of the \emph{Erd\"{o}s-Dushnik-Miller theorem}.

The range of applicability of Galvin's property has not been limited to the borders of infinite combinatorics. There is also a recently-opened avenue of research in \emph{Prikry-type forcings} where  Galvin's property is being used to classify  the intermediate models of the Prikry and Magidor generic extensions \cite{Parttwo}, it is equivalent to the density of old sets in Prikry extensions \cite{GitDensity,bgp}, and violating Galvin's property of $U$ is necessary for the corresponding Prikry forcing to project onto Cohen forcing $\mathrm{Add}(\kappa,\kappa^+)$ \cite{OnPrikryandCohen}. As a final application, we mention that the Galvin property serves as a yardstick principle in canonical inner models to separate certain large cardinal notions \cite{SatInCan}.

\subsection{Main results}\label{Subsec:IntroResults}
The first  significant observation of this paper is that the Galvin property is deeply connected to the Tukey order.

\begin{thm.2.4}
Let $\lambda\le \kappa$ be cardinals, where $\lambda$ is inaccessible, 
and let  $U$ be a $\lambda$-complete ultrafilter over $\kappa$.
Then $(U,\supseteq)\equiv_T ([2^\kappa]^{<\lambda},\subseteq)$ if and only if $\neg \Gal(U,\lambda,2^\kappa)$.
In particular, 
if $\kappa$ is a measurable cardinal then the  non principal
$\kappa$-complete ultrafilters over $\kappa$ which are Tukey-top\footnote{Namely, belongs to the maximum Tukey class among $\kappa$-directed closed posets of cardinality $2^\kappa$.}  are exactly those for which $\neg \Gal(U,\kappa,2^\kappa)$ holds.
\end{thm.2.4}

Isbell  \cite{Isbell65}  and Juh\'{a}sz \cite{Juhasz67} independently constructed, in \textsf{ZFC}, 
ultrafilters which are Tukey-top among directed partial orders of size continuum.
It turns out that this construction generalizes to uncountable cardinals  if one assumes the filter extension property. Moreover, it is possible to combine this  construction with the methods from 
 \cite{SatInCan} using $\Diamond(\kappa)$ to obtain a $q$-point which is a Tukey-top. This asserts the conjecture from \cite{SatInCan}: 

\begin{cor2.23}
  If $\kappa$ is $\kappa$-compact, then there is a Tukey-top ultrafilter among the $\kappa$-complete ultrafilters which is a $q$-point.
\end{cor2.23}

Section 3 maps out the connections between the spectra of characters for filter bases, the Tukey orders with respect to both $\contains$ and $\contains^*$, 
and 
the Galvin property with respect to both of these orders. Also, the analog $\Gal^*(U,\kappa,\lambda)$ of $\Gal(U,\kappa,\lambda)$ is introduced (see Definition \ref{def: Gal star}), which guarantees that from any given $\lambda$-many sets in $U$ there are $\kappa$-many of them which have a pseudo intersection in $U$. We then use it to extend and improve some results of Milovich in \cite{Milovich08}, including:
\begin{thm.316}
    Let $\kappa$ be either $\omega$ or a measurable cardinal. For any $\kappa$-complete ultrafilter $U$ over $\kappa$, $( U,\supseteq^*)\equiv_T( [2^\kappa]^{<\kappa},\subseteq)$ if and only if $\neg\Gal^*(U,\kappa,2^\kappa)$. 
\end{thm.316}
\begin{thm.317}
    Suppose that $\kappa$ is $\kappa$-compact. Then there is a $q$-point ultrafilter $U$ such that $\neg \Gal^*(U,\kappa,2^\kappa)$.
\end{thm.317}

Given ultrafilters $U\ge_T V$ on $\kappa$, a priori there are $2^{2^{\kappa}}$ many cofinal maps  from $U$ to $V$.
A valuable mechanism 
which unlocked much of the 
Tukey theory 
for ultrafilters
over $\omega$ was the discovery by 
Dobrinen and Todorcevic \cite{Dobrinen/Todorcevic11}
that any cofinal map from a p-point $U$ into an ultrafilter $V$ is actually continuous upon restriction to a suitably chosen member of $U$.
In Theorem \ref{Thm:DT20}, we
derive the analogue of this theorem for p-points over  measurable cardinals, 
and derive stronger forms of continuity  for  normal ultrafilters (Corollary \ref{cor.4.2}).

Furthermore, Theorem \ref{Prodthm.20DT-kappa}
shows that cartesian products of finitely many p-points  over a measurable cardinal $\kappa$ also have continuous cofinal maps.
Whereas such a theorem for ultrafilters over $\om$ is only useful for rapid p-points (see Section 4 of \cite{Dobrinen/Todorcevic11}),
surprisingly Theorem \ref{Prodthm.20DT-kappa} 
becomes a central tool for investigating cofinal types of Fubini products of p-points, due to Theorem \ref{ProductCor} below.
The results in Section $4$ 
 are used throughout the rest of the paper  to bound the number of Tukey-predecessors of a $p$-point, and to obtain models where there are many Tukey-incomparable normal ultrafilters.


In Section $5$ we turn our attention to the Tukey-class of sums, products and powers of ultrafilters.
A number of surprising differences between the Tukey theory of ultrafilters over  measurable cardinals  versus  $\om$ have emerged.
On $\om$,
 Dobrinen and Todorcevic \cite{Dobrinen/Todorcevic11} 
 proved that if $U$ is a rapid $p$-point, then for any ultrafilter $V$, $V\cdot U\equiv_T V\times U$ and in particular any Fubini iterate of $U$ is Tukey equivalent to $U$. Moreover, Dobrinen and Todorcevic presented a non-rapid $p$-point $U$ such that $U<_TU^2$. On measurable cardinals we have found the situation to be  quite different:

\begin{cor.5.6}
Given $\kappa$ measurable, 
    for every two $\kappa$-complete  ultrafilters $U,V$ over $\kappa$, $U\cdot V\equiv_T U\times V$.
    \end{cor.5.6}

    \begin{cor.5.7}
    Given $\kappa$ measurable, 
 for any $\kappa$-complete ultrafilter $U$ over  $\kappa$, $U^n\equiv_T U$.
\end{cor.5.7}

In particular, we can conclude that the class of  $\kappa$-complete Galvin ultrafilters is closed under finite Fubini powers.

In an attempt to approximate the class of Galvin ultrafilters (i.e.\ non-Tukey-top), we define in Section $6$ the notions of {\em basic} and {\em basically generated} ultrafilter. These notions  generalize the  notions for ultrafilters on $\om$ from \cite{Dobrinen/Todorcevic11} with the slight modification of uniformity. We  prove that basic is equivalent to $p$-point
(Propositions 
\ref{Prop:ppt implies basic} and \ref{Prop:Basicimpliesppt})
and 
give a shorter proof of a result in \cite{Parttwo} 
that 
$p$-point implies Galvin. 
More generally,  we  prove that $p$-point implies basically generated which in turn implies Galvin (Theorem \ref{thm.bgnottop}). Finally, we prove that the class of basically generated ultrafilters is closed under sums (Theorem \ref{thm.bgclosedclass}).
This in particular implies that sums (and limits) of $p$-points are Galvin, recovering  the main result from \cite{SatInCan}.

In Section $7$, we present several classical models and study the Tukey-order in those models. The first model is $L[U]$, in which Kunen \cite{Kunen1970} proved that every $\sigma$-complete ultrafilter $W$ over $\kappa$ is Rudin-Keisler equivalent to a finite Fubini power of $U$. The results of Section 6 imply that:

\begin{thm.7.2} In $L[U]$, the $\sigma$-complete ultrafilters over $\kappa$ form a single Tukey class which is not the Tukey-top one, and is the union of $\omega$-many Rudin Keisler equivalence classes.
\end{thm.7.2}

We then prove that the Mitchell order is orthogonal to the reverse Tukey order (Theorem \ref{Mitchellorthogonal}) and use it to construct models with 
chains of Tukey types of various lengths.
 Finally, we consider the situation in the Kunen-Paris \cite{Kunen-Paris} model where there are $2^{(2^\kappa)}$-many distinct normal ultrafilters and use our results from Section $4$ to prove:

\begin{thm.7.9}
      Assume \textsf{GCH}
      and that $\kappa$ is a measurable cardinal. Then there is a generic extension where \textsf{GCH} holds and there is a Tukey-chain of order type $\omega+1$.
\end{thm.7.9}

One problem we did not settle in this paper, is the problem of minimality in the Tukey order. In \cite{Raghavan/Todorcevic12}, Todorcevic proved that Ramsey ultrafilters are minimal in the Tukey order. We conjecture that a similar situation holds for uncountable cardinals and we call that conjecture \textit{the minimality conjecture}. Given the minimality conjecture, we use Goldberg's \textit{Ultrapower Axiom} (\textsf{UA}) \cite{GoldbergUA} to fully characterize the Tukey order in models of \textsf{UA} below a $\mu$-measurable cardinal. In particular, this characterization will hold true for the Mitchell models $L[\vec{U}]$:

\begin{thm.UA}
  Assume \textsf{UA} and the General Minimality Conjecture. Suppose that $\kappa$ is a measurable cardinal with $o(\kappa)<2^{(2^\kappa)}$. Then the Tukey types  of ultrafilters on $\kappa$ is isomorphic to $([o(\kappa)]^{<\omega},\subseteq)$.
    \end{thm.UA}

\section{Definitions and preliminaries}

We begin by reviewing some  standard notation used throughout.
Let $\kappa$ be a cardinal and $X$ be any set. Then $[X]^\kappa=\{Y\in P(X)\mid |Y|=\kappa\}$ and $[X]^{<\kappa}=\{Y\in P(X)\mid |Y|<\kappa\}$.
Let $\kappa$ be regular. For two subsets of $\kappa$, we write  $X\subseteq^* Y$ to denote that  $X\setminus Y$ is bounded in $\kappa$. Similarly, for $f,g:\kappa\rightarrow\kappa$ we denote $f\leq^* g$ if there is $\alpha<\kappa$ such that for every $\alpha\leq\beta<\kappa$, $f(\beta)\leq g(\beta)$.  We say that $C\subseteq \kappa$ is {\em closed and unbounded  at $\kappa$ (club)}, if it is a closed subset with respect to the order topology on $\kappa$ and unbounded in the ordinals below $\kappa$. The {\em club filter} is the filter:
$$\Cub_\kappa:=\{X\subseteq \kappa\mid \exists C\subseteq X, \ C\text{ is a club at }\kappa\}.$$

Given two partial orders $( P,\leq_P)$ and $( Q,\leq_Q)$, their \textit{(cartesian) product} is  denoted by  $( P,\leq_P)\times( Q,\leq_Q)=( P\times Q,\leq_{P\times Q})$ where $(p,q)\leq_{P\times Q}(p',q')$ iff both $p\leq_P p'$ and $q\leq_{Q}q'$.

\subsection{Background and basic results for the Tukey order on ultrafilters}
Let $U$ and $V$ denote ultrafilters on an infinite cardinal $\kappa$.  We shall assume throughout that all ultrafilters are uniform.
Unless explicitly stated otherwise, each ultrafilter is assumed to be partially ordered by reverse inclusion.
A map $f:U\rightarrow V$ is {\em monotone} if  $f(A)\contains f(B)$ whenever $A\contains B$.
The following simple but useful fact is a straightforward generalization 
of Fact 6 in \cite{Dobrinen/Todorcevic11}.

\begin{fact}\label{fact.6DT}
Let $U$ and $V$ be ultrafilters on any infinite cardinal $\kappa$.
If $V\le_T U$, then this is witnessed by a monotone cofinal map.
\end{fact}

The following connection between the Tukey and Rudin-Keisler orders is well-known:
Given ultrafilters
$U,V$  over base sets $X,Y$, respectively,  we say that  
$U$ is {\em Rudin-Keisler below}  $V$, and write $U\leq_{RK} V$, if  there is a function $f:Y\rightarrow X$ such that $f_*(V)=U$, where $$f_*(V):=\{B\subseteq X\mid f^{-1}[B]\in V\}$$   
If $U\leq_{RK}V$ and $V\leq_{RK}U$ then we denote this by $V\equiv_{RK}U$, and in this case there is a bijection $f:X\rightarrow Y$ such that $f_*(U)=V$. Note that if $f_*(U)=V$, then the map $g:U\rightarrow V$ defined by $g(X)=f''X$ is a cofinal map. Hence: 
\begin{proposition}
    If $U\leq_{RK} V$ then $U\leq_T V$.
\end{proposition}

It follows that cofinal types are a coarsening of the Rudin-Keisler equivalence classes of ultrafilters.

\begin{definition}
For any  ultrafilter $U$ over $X$ and a sequence $\l V_x\mid x\in X\r$ of ultrafilters over $Y$, define the {\em$U$-sum of the ultrafilters $V_x$}, to be the filter over $X\times Y$ defined as follows:
\begin{equation*}
    \sum_{U}V_x:=\Bigg\{A\subseteq X\times Y
: \ \Big\{x\in X\mid \{y\in Y\mid \l x,y\r\in B\}\in V_x\Big\}\in U\Bigg\}
\end{equation*}
If there is an ultrafilter $V$ such that for every $x\in X$, $V_x=V$, we call the ultrafilter $U\cdot V:=\sum_UV$ the {\em Fubini product of $U$ and $V$}, or simply {\em product}.
\end{definition}
In this paper we will also consider the Cartesian product $U\times V$, and we emphasize here the distinction between $U\cdot V$ and $U\times V$.
\begin{proposition}\label{product}
For any ultrafilters $U,V$, $U\times V$ is the minimal Tukey-type extension of both $U$ and $V$. 
\end{proposition}
\begin{proof}
First, note that the projection maps witness that $U\times V\geq_{T}U,V$. These are projections since any cofinal set in $U\times V$ must be cofinal in both coordinates. Suppose that $W\geq_T U,V$ with witnessing cofinal maps $f:W\rightarrow U$ and $g:W\rightarrow V$. Then $h:W\rightarrow U\times V$ define by $h(X)=\l f(X),g(X)\r$ is cofinal.
\end{proof}
\begin{corollary}\label{productAndCart}
For every two ultrafilters $U,V$, $U\times U\equiv_T U$ and $ U\cdot V \geq_T U\times V$.
\end{corollary}

\begin{proof}
    It is well known that the projection to the coordinates is a Rudin-Keisler projection of $U\cdot V$ to $U$ and $V$, hence $U\cdot V\geq_{T} U,V$ and by minimality $U\cdot V\geq_{T} U\times V$. 
\end{proof}
\subsection{Some definitions of ultrafilters over a measurable cardinal}

\begin{definition}
Let $U$ be an ultrafilter over a 
 regular cardinal $\kappa$.
\begin{enumerate}
    \item $U$ is {\em uniform } if for every $X\in U$, $|X|=\kappa$.
    \item $U$ is {\em $\lambda$-complete} if $U$ is closed under intersections of less than $\lambda$ many of its members.
    \item $U$ is {\em normal} if $U$ is closed under diagonal intersection, i.e.\ if $\l A_i\mid i<\kappa\r\subseteq U$, then $\Delta_{i<\kappa}A_i:=\{\nu<\kappa\mid \forall i<\nu\, (\nu\in A_i)\}\in U$.
    \item $U$ is {\em Ramsey} if for any function $f:[\kappa]^2\rightarrow 2$ there is an $X\in U$ such that $f\restriction [X]^2$ is constant. 
    \item $U$ is {\em selective} if for every function $f:\kappa \rightarrow \kappa$, there is an $X \in U$ such that
 $f\restriction X$ is  either constant or one-to-one.
\item (Kanamori \cite{KanPPoint}) $U$ is {\em rapid} if for each normal function $f: \kappa \rightarrow \kappa$ (i.e.\ increasing and continuous),
there exists an $X \in U$ such that
$\otp(X \cap f(\alpha)) \leq \alpha$ for each $\alpha< \kappa$.
\item $U$ is an {\em $r$-point} if whenever $f$ is continuous and almost injective 
(i.e.\ bounded pre-images), there is an $X\in U$ such
 that $|f^{-1}(\{\alpha\})|\le \alpha$ for every $\alpha< \kappa$.
\item $U$ is a {\em $p$-point} if whenever $f:\kappa\rightarrow\kappa$ is not constant on a set in $U$, there is an $X\in U$ such that for every $\gamma<\kappa$,  $|f^{-1}[\gamma]\cap X|<\kappa$.
\item $U$ is a {\em $q$-point} 
if every  function $f:\kappa\ra\kappa$  which is almost 1-1 (mod $ U $) is injective (mod $ U $).
\end{enumerate}

\end{definition}
\begin{proposition}
    \begin{enumerate}
        \item Normal $\Rightarrow$ Ramsey.
        \item Ramsey $\Leftrightarrow$ Selective.
        \item $U$ is $p$-point iff every sequence $\l X_i\mid i<\kappa\r\in [U]^\kappa$ admits a pseudo-intersection, namely, there is $X\in U$ such that for every $i<\kappa$, $X\subseteq^* X_i$.        \item $U$ is a $q$-point iff $U$ is Rudin-Keisler equivalent to an ultrafilter $W$ such that $Cub_\kappa\subseteq W$.
    \end{enumerate}
\end{proposition}
For more information about these definitions and implications we refer the reader to \cite{Kanamori1976UltrafiltersOA} or \cite{TomTreePrikry}.
\begin{proposition}
Selective $\Leftrightarrow$ $p$-point and $q$-point $\Leftrightarrow$  RK-minimal among uniform ultrafilters $\Leftrightarrow$ RK-equivalent to a normal. 
\end{proposition}
\begin{proof}
See (\cite{Kanamori1976UltrafiltersOA} or \cite{TomTreePrikry} or \cite{GoldbergUA}).
\end{proof}
\begin{theorem}{\rm(Kanamori) \cite[Theorem 1.3]{KanPPoint}}
 
\begin{enumerate}
    \item If $\kappa=\omega$ then rapid $\Leftrightarrow$ $r$-point.
    \item if $\kappa>\omega$ then rapid $\Leftrightarrow$ $q$-point.
\end{enumerate}
\end{theorem}

The original definition of rapid  ultrafilter on $\omega$, obviously, does not mention normality of the function. 
However, for measurable cardinals the normality assumption cannot be dropped:
\begin{proposition}
Suppose that $U$ is a normal ultrafilter over a measurable cardinal $\kappa$.
Then for any increasing function $f:\kappa\rightarrow \kappa$  such that $f(\alpha)>\alpha$ for every $\alpha$ (in particular not continuous) and for any $X\in U$, we have that
$$\{\alpha<\kappa\mid \otp(X\cap f(\alpha))\geq \alpha+1\}\in U$$ 
\end{proposition}
\begin{proof}
Let $f,X$ be as in the proposition.
Normality is equivalent to $[\id]_U=\kappa$.
Also, since $U$ is uniform, every $X\in U$ satisfies that $\otp(X)=\kappa$.
Let $j_U:V\rightarrow M_U$ be the ultrapower embedding by the normal measure $U$. Then  we have the following equality (just for normal):
$$Z\in U\Leftrightarrow \kappa\in j_U(Z)$$
In particular, since $X\in U$, we know that $\kappa\in j_U(X)$.
Notice that since $\kappa$ is the critical point of the embedding $j_U$ (namely $j_U\restriction \kappa=\id$ and $j_U(\kappa)>\kappa$) and since $X\subseteq\kappa$ we have that $j_U(X)\cap\kappa=X$.

By elementarity, and by the assumption on $f$, $j_U(f):j_U(\kappa)\rightarrow j_U(\kappa)$ and $j_U(\kappa)>\kappa$, so we my apply $j_U(f)$ to $\kappa$ and by the assumption $j_U(f)(\kappa)>\kappa$. In particular,
$$j_U(X)\cap j_U(f)(\kappa)\supseteq j_U(X)\cap \kappa+1=X\cup\{\kappa\}$$
In particular, $$(*)\ \ \otp(j_U(X)\cap j_U(f)(\kappa))\geq\kappa+1>\kappa$$

Consider the set $A=\{\alpha<\kappa\mid \otp(X\cap f(\alpha))>\alpha\}$, we want to prove that $A\in U$, or equivalently, that $\kappa\in j_U(A)$. By elementarity, the set $j_U(A)=\{\alpha<j_U(\kappa)\mid \otp(j_U(X)\cap j_U(f)(\alpha))>\alpha\}$ but this is exactly what we checked in $(*)$.
\end{proof}
The above also proves that $p$-points are not rapid in this sense.

\section{The maximum Tukey class and the Galvin property}

In this  section, we lay forth the   connections between the Galvin property and the Tukey order. 
Recall that a partial ordering  $(P,\le_P)$ is {\em $\lambda$-directed closed} if  every  subset of $P$ of size less than $\lambda$ has an upper bound in $P$.

\begin{theorem}\label{directedposet}
Let $\lambda$ and $\mu$ be any cardinals. Suppose that $(P,\leq_P)$ is $\lambda$-directed closed  and that $|P|\leq\mu$.
Then $(P,\leq_P)\leq_T([\mu]^{<\lambda},\subseteq)$. In other wards, $([\mu]^{<\lambda},\subseteq)$ is maximal in the Tukey order among $\lambda$-directed closed posets of cardinality at most $\mu$.
\end{theorem}

\begin{proof}
First note that if $\mu<\lambda$ and $P$ is a $\lambda$-directed poset of cardinality at most $\mu$, then $P$ has a greatest element and the theorem is trivially true. Hence, let us assume that $\lambda\leq\mu$. Fix any one-to-one function $g:P\rightarrow \mu$,  and let us prove that the function $f(p)=\{g(p)\}$ is an unbounded function from $P$ to $[\mu]^{<\lambda}$. Suppose that $X$ is unbounded. Then $|X|\geq \lambda$ (by $\lambda$-directedness) and therefore $g''X$ has size at least $\lambda$. This implies that $\cup f''X=g''X$ has size $\lambda$. If follows that $f''X$ is unbounded in $[\mu]^{<\lambda}$.
\end{proof}
Any order $(P,\leq_P)$ such that $(P,\leq_P)\equiv_T([\mu]^{<\lambda},\subseteq)$ is called {\em Tukey-top}.
It is important that we only restrict ourselves to $\lambda$-directed closed posets, as the following proposition suggests:
\begin{proposition}
    If $(P,\leq_P)$ is not $\lambda$ directed-closed and $(Q,\leq_Q)$ is $\lambda$-directed closed then $\neg( (P,\leq P)\leq_T (Q,\leq_Q))$
\end{proposition}
\begin{proof}
    Suppose otherwise that there is an unbounded function $f:P\rightarrow Q$. Since $P$ is not $\lambda$-directed closed, there is a family $B:=\{p_i\mid i<\theta\}$ for $\theta<\lambda$ which is unbounded. Thus $f''B\subseteq Q$ is  unbounded. Since $Q$ is $\lambda$-directed, and $f''B$ is a family of size less than $\lambda$,  there is $q\in Q$ which bounds $f''B$, contradiction.
\end{proof}
So, for example, ultrafilters which are not $\sigma$-complete cannot be Tukey below ultrafilters which are $\kappa$-complete. 
The other direction is simply not true in general since for example $\omega\times\omega_1\geq_T \omega_1$. More generally, we have that:
\begin{fact}
    Suppose $(P,\leq_P)$ is not $\lambda$-directed closed, then for every $(Q,\leq_Q)$, $(P,\leq_P)\times (Q,\leq_Q)$ is not $\lambda$-directed closed.
\end{fact}
The Galvin property (Definition \ref{Def:Galvin Property}) turns out to 
provide an equivalent condition for an ultrafilter to be Tukey-top.



\begin{theorem}\label{GlavinMainTheorem}
Let $\lambda\le \kappa$ be cardinals, 
and let  $U$ be a $\lambda$-complete filter over $\kappa$.
Then $(U,\supseteq)\geq_T ([\kappa]^{<\lambda},\subseteq)$ if and only if $\neg \Gal(U,\lambda,\kappa)$.
In particular, if 
$\kappa$ is a measurable cardinal then the $\kappa$-complete ultrafilters over $\kappa$ which are Tukey-top are exactly those for which $\neg \Gal(U,\kappa,2^\kappa)$ holds.
\end{theorem}

\begin{proof}
Since $(U,\supseteq)$ is $\lambda$-directed whenever $U$ is $\lambda$-complete, and since $|U|=2^\kappa$, Theorem \ref{directedposet} 
applies ($\lambda$ being itself and $\mu$ being $2^\kappa$) and yields
$(U,\supseteq)\leq_T ([2^\kappa]^{<\lambda},\subseteq)$. So the ``In particular" part of the theorem follows from the first part. 

Suppose that $\l A_i\mid i<\kappa\r$ witnesses that $\neg \Gal(U,\lambda,\kappa)$. 
To show that $(U,\contains)\ge_T ([\kappa]^{<\lambda},\subseteq)$,
define $f:[\kappa]^{<\lambda}\rightarrow U$
by $f(I)=\bigcap_{i\in I}A_i$, 
and let us prove that $f$  is an unbounded map: 
Suppose that $X\subseteq [\kappa]^{<\lambda}$ is unbounded. Then $|\cup X|\geq\lambda$. 
 Now
 $$\bigcap f''X=\bigcap\{\cap_{i\in I}A_i\mid I\in X\}=\bigcap_{i\in\cup X}A_i.$$ By the assumption that the sequence witnesses the failure of Galvin's property, it follows that $\bigcap f''X\notin U$ and therefore $f''X$ is  unbounded in $U$.

For the other direction, suppose $g:[\kappa]^{<\lambda}\rightarrow U$ is unbounded, and let us show that $\neg\Gal(U,\lambda,\kappa)$. The witnessing sequence will be $\l X_\alpha\mid\alpha<\kappa\r$, where $X_\alpha=g(\{\alpha\})$. Let us prove that $X_\alpha$ is as wanted. Fix any $I\in[\kappa]^{\lambda}$. Then the set $\hat{I}=\{\{i\}\mid i\in I\}\subseteq [\kappa]^{<\lambda}$ is unbounded. By our assumption, $g''\hat{I}=\{X_i\mid i\in I\}$ is also unbounded, it follows that $\bigcap_{i\in I}X_i\notin U$. 
\end{proof}


Galvin ultrafilters have been studied extensively in recent years. 
The most recent sufficient condition for an ultrafilter to be Galvin is the following 
 \cite{SatInCan}:

\begin{corollary}
Let $\kappa$ be a measurable cardinal and suppose that $W$ is Rudin-Keisler equivalent to an ultrafilter of the form:
$$\sum_{U}\sum_{U_{\l\alpha_1\r}}\dots\sum_{\l\alpha_1,\dots,\alpha_{n-1}\r}U_{\l\alpha_1,\dots,\alpha_{n}\r}$$
where
each of the ultrafilters $U_{\l\alpha_1,..,\alpha_k\r}$ is a $p$-point ultrafilter. Then $\Gal(W,\kappa,2^\kappa)$, and in particular,  $W$ is not Tukey-top.
\end{corollary}

Another type of ultrafilters which have the Galvin property are fine $\kappa$-complete ultrafilters over $[\lambda]^{<\kappa}$:
\begin{definition}
    An ultrafilter $U$ over $[\lambda]^{<\kappa}$ is called \textit{fine} if for every $\alpha<\lambda$, $A_\alpha:=\{X\in[\lambda]^{<\kappa}\mid \alpha\in X\}\in U$.
\end{definition}
The assumption that there is a fine $\kappa$ complete ultrafilter over $\lambda$ is known to be much stronger than measurability.
\begin{definition}
    A cardinal $\kappa$ is called \textit{$\lambda$-strongly-compact}  if $[\lambda]^{<\kappa}$ carries a fine $\kappa$-complete ultrafilter.
\end{definition}
It is immediate that if $U$ is a fine ultrafilter over $[\lambda]^{<\kappa}$, then the sequence $\l A_\alpha\mid \alpha<\lambda\r$ witnesses that $\neg \Gal(U,\kappa,\lambda)$. In particular, if $\lambda=2^\kappa$, and $U$ is a fine $\kappa$-complete ultrafilter over $[2^\kappa]^{<\kappa}$, then by Theorem \ref{GlavinMainTheorem}, $(U,\supseteq)\equiv_T ([2^{(2^\kappa)}]^{<\kappa},\subseteq)$. It follows that $(U,\supseteq)$ is maximal among $\kappa$-directed posets of cardinality at most $2^{(2^\kappa)}$ and in particular among $\kappa$-complete ultrafilters over $\kappa$. \begin{corollary}
        Suppose that $\kappa$ is $2^\kappa$-strongly compact and $U$ is a fine $\kappa$-complete ultrafilter over $[2^\kappa]^{<\kappa}$ witnessing that. Then for every $\kappa$ complete ultrafilter $W$ over $\kappa$, $(W,\supseteq)\leq_T(U,\supseteq)$ \end{corollary}

\subsection{Generalization of Isbell's result}
It turns out that on $\omega$, there is a \textsf{ZFC} construction which outright gives a Tukey-top ultrafilter:
\begin{theorem}[Isbell \cite{Isbell65}]
    There is a ultrafilter $U$ on $\omega$ such that $U\equiv_T ([\mathfrak{c}]^{<\omega},\subseteq)$.
\end{theorem}
The long-standing open question in this direction is the following:
\begin{question}[Isbell, 1965]
    Is it consistent that for every ultrafilter $U$ on $\omega$, $U\equiv_T ([\mathfrak{c}]^{<\omega},\subseteq)$? equivalently, that $\neg \Gal(U,\omega,\mathfrak{c})$?
\end{question}
On measurable cardinals, there is always a normal ultrafilter, and therefore the analog of this question to yield an ultrafilter $U$ such that $\Gal(U,\kappa,2^\kappa)$ over a measurable cardinal $\kappa$ is not very interesting. However, the analog of Isbell's construction does not generalize automatically for any measurable cardinal as there are models (e.g. $L[U]$--see Theorem \ref{LofU}) where every ultrafilter satisfies $\Gal(U,\kappa,2^\kappa)$.

Before entering the realm of large cardinals, let us start this subsection with a generalization of Isbell's construction to any  two cardinals $\mu\leq\kappa$. As in the construction of \cite{SatInCan} and \cite{Non-GalvinFil}, the construction uses the notion of a $\kappa$-independent family.

\begin{definition}
A $\kappa$-independent family is a family $\l A_i\mid i<\lambda\r$ of subsets of $\kappa$ such that for every $I,J\in[\lambda]^{<\kappa}$ such that $I\cap J=\emptyset$, we have $(\bigcap_{i\in I}A_i)\cap(\bigcap_{j\in J}A^c_j)\neq\emptyset$.
\end{definition}
 Hausdorff \cite{Hausdorff1936} proved that any regular cardinal $\kappa$ satisfying $\kappa^{<\kappa}=\kappa$ admits a $\kappa$-independent family of maximal length (see also \cite[Exercise (A6), Chapter VIII]{Kunen}). Even without the above cardinal arithmetic assumption, one can force such a family by adding mutually generic Cohen sets (see for example \cite[Proposition 4.9]{SatInCan}).

\begin{proposition}\label{the NonGlavinFilter}
Fix a $\kappa$-independent family $\vec{A}=\l A_i\mid i<\lambda\r$ and let $\omega\leq\mu\leq \kappa$ be any cardinal. Then the set $\{ (\bigcup_{i\in I}A_i)\setminus (\bigcup_{j\in J}A_j)\mid |I|=\mu\text{ and } |J|<\mu\}$ has the $<\mu$-intersection property. In particular, this set generates a $\mu$-complete filter $\mathcal{F}^{\mu}_{\vec{A}}$ over $\kappa$.\end{proposition}

\begin{proof}
Let $\l I_\alpha,J_\alpha\mid \alpha<\rho\r$ for $\rho<\mu$ be such that $|I_\alpha|=\mu>|J_\alpha|$. We want to prove that $\bigcap_{\alpha<\rho}\Big((\bigcup_{i\in I_\alpha}A_i)\setminus(\bigcup_{j\in J_\alpha}A_j)\Big)\neq\emptyset$. 
This follows by the $\kappa$-independence of the family.
\end{proof}

\begin{theorem}\label{TheNonGlavintheorem}
If $\mathcal{F}^{\mu}_{\vec{A}}\subseteq U$ is an ultrafilter, then the family $\l (\kappa\setminus A_i)\mid i<\lambda\r$ witnesses the fact that $\neg \Gal(U,\mu,\lambda)$.
\end{theorem}

\begin{proof}
Clearly each $\kappa\setminus A_i$ is in $\mathcal{F}^{\mu}_{\vec{A}}$. To see $\neg \Gal(U,\mu,\lambda)$, let $I\in[\lambda]^\mu$, then \begin{equation*}
    \bigcap_{i\in I}(\kappa\setminus A_i)=\kappa\setminus(\bigcup_{i\in I}A_i)
\end{equation*} Since $|I|=\mu$, $\bigcup_{i\in I}A_i\in \mathcal{F}^\mu_{\vec{A}}$, and thus, also in $U$. It follows that $\bigcap_{i\in I}(\kappa\setminus A_i)\notin U$.
\end{proof}
\begin{definition}
    A cardinal $\mu$ is said to have the {\em$\kappa$-filter extension property} if every $\mu$-complete filter over $\kappa$ can be extended to a $\mu$-complete ultrafilter.
\end{definition}
The axiom of choice implies that  $\omega$ has the $\kappa$-filter extension property for every cardinal $\kappa$. A cardinal $\kappa$ which has the $\kappa$-filter extension property is called {\em $\kappa$-compact}. The assumption that $\mu$ has the $\kappa$-filter extension property  is a consequence of $\mu$ being a $\kappa^+$-$\Pi^1_1$-subcompact cardinal which in turn is below a supercompact cardinal \cite{Hayut2018PartialSC}.  For more information about the filter extension property see \cite{filterextension} or \cite{Hayut2018PartialSC}.
\begin{corollary}\label{kappacompactmaximal}
Suppose that $\kappa^{<\kappa}=\kappa$. If $\mu$ has the $\kappa$-filter extension property, then there is a $\mu$-complete ultrafilter $U$ over $\kappa$ such that $\neg \Gal(U,\mu,2^\kappa)$. In particular, there is a Tukey-top ultrafilter among the $\mu$-complete ultrafilters.
\end{corollary}
\begin{proof}
    Since $\kappa=\kappa^{<\kappa}$, there is a $\kappa$-independent family $\vec{A}$ of size $2^\kappa$, and the $\mu$-complete filter $\mathcal{F}^\mu_{\vec{A}}$ of Proposition \ref{the NonGlavinFilter} exists. Applying the $\kappa$-extension property to $\mathcal{F}^\mu_{
    \vec{A}}$ guarantees the existence of a $\mu$-complete ultrafilter $U$ which extends $\mathcal{F}^\mu_{\vec{A}}$. By Theorem \ref{TheNonGlavintheorem} we have that $\neg \Gal(U,\mu,2^\kappa)$. 
\end{proof}
In \cite{SatInCan}, the first author proved that if $\kappa$ is $2^\kappa$-supercompact, then there is a $\kappa$-complete ultrafilter 
 $\Cub_\kappa\subseteq U$ which is non-Galvin (namely $\neg \Gal(U,\kappa,\kappa^+)$). This raises
 the question \cite[Question 5.3]{SatInCan} whether there is a weaker large cardinal  notion
 which guarantees the existence of such a non-Galvin ultrafilter. Moreover, in that paper, it was conjectured that $\kappa$-compact cardinals suffice. The following corollary asserts that conjecture:
\begin{corollary}
    If $\kappa$ is $\kappa$-compact, then there is a $\kappa$-complete ultrafilter $U$ such that $\neg \Gal(U,\kappa,2^\kappa)$.
\end{corollary}
\begin{corollary}\label{omegaiscompact}
    For any infinite cardinal  $\kappa$, there is always an ultrafilter $U$ over $\kappa$ such that $\neg \Gal(U,\omega,2^\kappa)$. 
\end{corollary}
Note that the ultrafilter $U$ above does not necessarily extend the club filter (i.e.\ $q$-point). Since the club filter is only a filter on cardinals of uncountable cofinalities, henceforth we will assume that
$\cf(\kappa)>\omega$. Let us show that a similar method as in \cite{SatInCan}, gives such an ultrafilter. For this we need the definition of a normal independent family:
\begin{definition}
 A sequence $\l A_i\mid i<\lambda\r$ is called a {\em normal} $\kappa$-independent family, if it is $\kappa$-independent and for any two disjoint subfamilies $\l A_{\alpha_i}\mid i<\kappa\r,\l A_{\beta_i}\mid i<\kappa\r\subseteq \l A_i\mid i<\lambda\r$, the diagonal intersection $\Delta_{i<\kappa}(A_{\alpha_i}\setminus A_{\beta_i})$ is a stationary subset of $\kappa$. 
\end{definition}
Distinguishing from $\kappa$-independent families, in order to guarantee the existence of a normal $\kappa$-independent family, more is needed. Hayut \cite{filterextension} used the following classical prediction principle:

\begin{definition}
    The \textit{diamond at $\kappa$} principle, denoted by $\Diamond(\kappa)$,  asserts the existence of a sequence $\l A_i\mid i<\kappa\r$ such that for every set $A\subseteq \kappa$, $\{\alpha<\kappa\mid A\cap \alpha=A_\alpha\}$ is stationary at $\kappa$.
\end{definition}
It is well-known that in $L$, $\Diamond(\kappa)$ holds for every regular cardinal. Moreover, Kunen and Jensen observed that if $\kappa$ is measurable (or even subtle) then $\Diamond(\kappa)$ holds.  The following proposition is due to Hayut\footnote{ Hayut uses the more general framework of stationary sets over $P_\kappa(\lambda)$ due to Jech \cite{Jech2010stationary}.}. For the proof see \cite[Proposition 4.2]{SatInCan}.
\begin{proposition}\label{normalfamily}
 If $\Diamond(\kappa)$ holds then there is a normal $\kappa$-independent family of length $2^\kappa$.
\end{proposition}
Next we claim that a suitable modification of the Isbell construction above will give rise to non-Galvin $q$-point ultrafilters.
\begin{proposition}
    Suppose that $\vec{A}:=\l A_i\mid i<\lambda\r$ is a normal $\kappa$-independent family, then the intersection of less than $\mu$-many sets  in $G_{\vec{A}}:=\{ (\bigcup_{i\in I}A_i)\setminus (\bigcup_{j\in J}A_j)\mid |I|=\mu, |J|<\mu\}$ is stationary. 
\end{proposition}
\begin{proof}
    See \cite[Corollary 4.4]{SatInCan}.
\end{proof}
\begin{proposition}
    Let $\vec{A}:=\l A_i\mid i<\lambda\r$ be a normal $\kappa$-independent family. Define 
    $$\mathcal{F}^{\nu,*}_{\vec{A}}=\{X\subseteq \kappa\mid \exists C\in \Cub_\kappa\  \exists B\in[G_{\vec{A}}]^{<\mu}\ C\cap(\bigcap B)\subseteq X\}$$
    Then $\mathcal{F}^{\nu,*}_{\vec{A}}$ is a $\min(\cf(\kappa),\mu)$-complete filter which extends the club filter and if $\mathcal{F}^{\nu,*}_{\vec{A}}\subseteq U$ is an ultrafilter, then $\l (\kappa\setminus A_i)\mid i<\lambda\r$ witnesses that $\neg \Gal(U,\mu,\lambda)$.
\end{proposition}

\begin{proof}
    The fact that this filter is $\min(\cf(\kappa),\mu))$-complete, follows by the known fact that the intersection of less than $\cf(\kappa)$-many clubs is a club and by definition of the previous proposition. The proof that $\neg \Gal(U,\mu,\lambda)$ holds is
    identical to the proof of Theorem \ref{TheNonGlavintheorem}.
\end{proof}

\begin{corollary}
    If $\mu$ has the $\kappa$-extension property, $\cf(\kappa)\geq\mu$ and $\Diamond(\kappa)$ holds, then there is a $\mu$-complete ultrafilter $U$ which extends the club filter and $\neg \Gal(U,\mu,2^\kappa)$.
\end{corollary}
\begin{corollary}
    If $V=L$, then for every regular cardinal $\kappa$, there is an ultrafilter (which is of course not even $\sigma$-complete) such that $\neg \Gal(U,\omega,2^\kappa)$ and $\Cub_\kappa \subseteq U$.
\end{corollary}
The following corollary is an answer to an analogous  question which is still open on  $\omega$. 
\begin{corollary}\label{compactQpoint}
    If $\kappa$ is $\kappa$-compact, then there is a Tukey-top among the $\kappa$-complete ultrafilters which is a $q$-point.
\end{corollary}
\section{Bases, the Tukey order, and the $\supseteq^*$ relation}
The ultrafilter number is a cardinal characteristic of the continuum which in its generalized form is defined as follows:
\begin{definition}
    Let $\kappa\geq\omega$ be any cardinal.\begin{enumerate}
        \item The {\em characteristic number} of a uniform ultrafilter $U$ is defined by $\Ch(U):=\min(|\mathcal{B}|\mid \mathcal{B}\text{ is a cofinal in }(U,\supseteq^*))$. 
        \item For any regular cardinal $\mu\leq\kappa$ the {\em $\mu$-spectrum} of $\kappa$ is denoted by $$\Sp_\mu(\kappa)=\{\Ch(U)\mid U\text{ is a }\mu\text{-complete uniform ultrafilter over }\kappa\}$$
        \item  The {\em $\mu$-ultrafilter number} of $\kappa$ is  $\mathfrak{u}_\mu(\kappa)=\min(\Sp_\mu(\kappa))$.
        \end{enumerate} 
\end{definition}
Note that requiring $\mathcal{B}$ to be a filter base (i.e. cofinal in $(U,\supseteq)$) instead of a cofinal subset of $(U,\supseteq^*)$, yield the same cardinals. Also note that
$$\mu_1<\mu_2\text{ implies }\Sp_{\mu_1}(\kappa)\supseteq \Sp_{\mu_2}(\kappa)\text{ and }\mathfrak{u}_{\mu_1}(\kappa)\leq\mathfrak{u}_{\mu_2}(\kappa)$$ 
\begin{definition}
Let $\kappa\geq\omega$ be any cardinal.
    \begin{enumerate}
         \item The {\em spectrum} of $\kappa$ is $\Sp(\kappa)=\cup_{\mu\in (\kappa+1)\cap Reg.}\Sp_\mu(\kappa)=\Sp_\omega(\kappa)$.
        \item The {\em ultrafilter number} of $\kappa$ is defined by 
        $\mathfrak{u}(\kappa)=\min(\Sp(\kappa))$. 
    \end{enumerate}
\end{definition}
Note that for any $\mu\leq\kappa$ regular, $\mathfrak{u}(\kappa)\leq\mathfrak{u}_\mu(\kappa)$. Kunen proved that $\mathfrak{u}(\omega)<2^\omega$ is consistent, and a major open problem is whether $\mathfrak{u}(\omega_1)<2^{\omega_1}$ is consistent. For measurable cardinals, Gitik and Shelah \cite{GSOnDO} proved that $\mathfrak{u}(\kappa)<2^\kappa$ is consistent from an extremely large cardinal---a huge cardinal. Later, the large cardinal assumption was improved by Brooke-Taylor, Fischer, Friedman, and Montoya \cite{BROOKETAYLOR201737} to a supercompact cardinal. Recently, a remarkable result of Raghavan and Shelah \cite{RagShel} established the consistency of 
$\mathfrak{u}(\kappa)<2^\kappa$ for $\kappa=2^{\aleph_0}$ starting with a much smaller large cardinal---a \textit{measurable cardinal}. However, in their model $2^{\aleph_0}$ is still very large. They also obtained the result on the much smaller cardinal $\aleph_{\omega+1}$ but starting again from a supercompact cardinal. The ultrafilter number on small cardinals such as $\omega_1$, or even on small measurable cardinals is not fully understood. Here we suggest some connections to the Tukey order. 

The natural homomorphism to consider suitable for comparing the characteristic number of two ultrafilters $U,V$ would be a function $f:U\rightarrow V$ such that $\mathrm{Im}(f)$ is cofinal in $V$. This requirement is a bit weaker than being a cofinal map. Nonetheless, if $f$ is a monotone cofinal map, which holds in many situations,
then requiring that $\mathrm{Im}(f)$ is cofinal is equivalent to being a cofinal map. 
\begin{proposition}
    Suppose that $U,W$ are any (not necessarily complete) ultrafilters such that $U\geq_T W$. Then $\Ch(U)\geq \Ch(W)$.
\end{proposition}
\begin{proof}
    Let $\{b_i\mid i<\Ch(U)\}$ be a filter base for $U$. Then $B=\{b_i\setminus \xi\mid i<\Ch(U),\xi<\kappa\}$ is cofinal in $U$. Let $f:U\rightarrow W$ be a cofinal map. Then $f''B$ is a filter base for $W$, and therefore $\Ch(W)\leq|f''B|\leq |B|=\Ch(U)$.
\end{proof}
Note that the other direction is not true in general, since for example if $U_0\triangleleft U_1$, $U_1$ is a $p$-point and $2^\kappa=\kappa^+$, then $\Ch(U_0)=\kappa^+=\Ch(U_1)$ but $U_0\not\equiv_{T}U_1$ (see Theorem \ref{Mitchellorthogonal}). 
As a corollary, we get that any 
    $\mu$-complete ultrafilter for which $\neg\Gal(U,\mu,2^\kappa)$ will satisfy $\Ch(U)=\max(\Sp_\mu(\kappa))$. In fact, we can say more:
\begin{proposition}
    Suppose that $U$ is an ultrafilter over $\kappa$, $\mu\leq\kappa$ and $\neg \Gal(U,\mu,2^\kappa)$ holds. Then $\Ch(U)=2^\kappa$.
\end{proposition}
\begin{proof}
     To see that $\Ch(U)=2^\kappa$, let $\l A_i\mid i<2^\kappa\r$ witness that $\neg \Gal(U,\mu,2^\kappa)$, and suppose toward a contradiction that there is cofinal set $\mathcal{B}\subseteq U$ such that $|\mathcal{B}|<2^\kappa$. In particular, for each $\alpha<2^\kappa$ there is $b_\alpha\in \mathcal{B}$ such that $b_\alpha\subseteq A_\alpha$. This defines a function from $2^\kappa$ to $\mathcal{B}$, and by the pigeonhole principle there is $b^*\in\mathcal{B}$ such that $I:=\{\alpha<2^\kappa\mid b_\alpha=b^*\}$ has size at least $\kappa$. It follows that $b^*\subseteq \cap_{\alpha\in I}A_\alpha$, and therefore $\cap_{\alpha\in I}A_\alpha\in U$, contradicting the fact that $\l A_\alpha\mid \alpha<2^\kappa\r$ witnesses the failure of $\Gal(U,\mu,2^\kappa)$.
\end{proof}
It is well known that if $\l A_i\mid i<\lambda\r$ is a maximal $\kappa$-independent family then $\l A_i\mid i<\lambda\r$ generates an ultrafilter $U$ such that $\Ch(U)=\lambda$. Since all the known constructions of non-Galvin ultrafilters rely on some independent family which witnesses the non-Galviness, the previous result is somehow not surprising. In this light, the following question seems natural:
\begin{question}
    Suppose that $U$ is a $\kappa$-complete ultrafilter over $\kappa$ such that $\neg \Gal(U,\kappa,\lambda)$. Is there a $\lambda$-independent family $\l A_i\mid i<\lambda\r\subseteq U$ which generates $U$? 
\end{question}

\begin{corollary}
    $\max(\Sp(\kappa))=2^\kappa$ is always realized by some (any) uniform ultrafilter $U$ such that $\neg \Gal(U,\omega,2^\kappa)$.
\end{corollary}
\begin{proof}
    By Corollary \ref{omegaiscompact}, such an ultrafilter $U$ always exists, and by the previous corollary, $2^\kappa=\Ch(U)$.
\end{proof}
We can now further exploit this connection between the Tukey order and the ultrafilter number as follows:
\begin{corollary}
    Suppose that $|\Sp_\mu(\kappa)|=\theta$. Then there are at least $\theta$-many $\mu$-complete ultrafilters over $\kappa$ with distinct Tukey classes.
\end{corollary}
There are known models due to Garti, Magidor, and Shelah where the spectrum of an infinite cardinal $\kappa$ has arbitrarily many cardinals (this requires  blowing  up $2^\kappa$ and large cardinals) \cite{GMSSpec}. As a corollary we get:
\begin{corollary}
    Assuming large cardinals, it is consistent that there are arbitrarily many distinct Tukey-classes of ultrafilters.
\end{corollary}
In a later section, we will present models where we have more control over how those classes are ordered.
 
 In order to compare the characteristics it suffices to compare the cofinal types of $(U,\supseteq^*)$. The next proposition suggests that, in some sense, the converse is also true.
\begin{proposition}
    Suppose that $U,W$ have generating sequences $\l X^U_i\mid i<\theta\r,\l X^W_i\mid i<\theta\r$ resp.\ which are $\supseteq^*$-increasing. Then $(U,\supseteq^*)\equiv_T (W,\supseteq^*)$.
\end{proposition}
\begin{proof}
   Let $f:U\rightarrow W$ be the map sending $f(X^U_i)=X^W_i$ and for $b\in U$, let $f(b)=f(X^U_i)$ for the minimal $i$ such that $X^U_i\subseteq^* b$.,
 Given a $\subseteq^*$-cofinal set $B$ in $U$,
    then for every $b\in B$,
    let $i_b$ be the minimal index such that $X^U_{i_b}\subseteq^* b$ (hence $f(g)=X^W_{i_b}$).
    We claim that $\{i_b\mid b\in B\}$ must be unbounded in $\theta$. Otherwise, there is $i^*$ which is a bound. In particular,  $X^U_{i^*}\subseteq^* b$ for every $b\in B$. Take any set $Y\in U$ such that $X^U_{i^*}\setminus Y$ is unbounded in $\kappa$ (just partition $X^U_{i^*}$ and take $Y$ to be the piece which belongs to $U$). Now there is $b\in B$ such that $b\subseteq^* Y$ and thus 
    $$b\subseteq^* Y\subseteq^* X^U_{i^*}\subseteq b$$
    Hence $b=^*Y=^*X^U_{i^*}$ which contradicts the fact that $X^U_{i^*}\setminus Y$ is unbounded in $\kappa$. 
    So we conclude that $\{i_b\mid b\in B\}$ is unbounded in $\theta$. By definition of $f$, we have that $f''B\subseteq \{X^W_i\mid i<\theta\}$ then is cofinal in $W$. It follows that $f$ is cofinal and $(U,\supseteq^*)\geq_T (W,\supseteq^*)$. The other direction is symmetric.
\end{proof}
\begin{remark}
    Ultrafilters with $\supseteq^*$-increasing generating sequences are forcible (see \cite{GSOnDO} or \cite{GalDet}).
\end{remark}
For the rest of this section, let us focus on the poset $(U,\supseteq^*)$.
The next lemma can be found in \cite[Lemma 2.19]{GalDet}:
\begin{lemma}\label{p-point}
     Suppose that $2^\kappa=\kappa^+$. Then $U$ is a $p$-point over $\kappa$ if and only if $U$ has a strong generating
 sequence\footnote{For the definition of a strong generating sequence see \cite[Def. 2.11]{GalDet}.} of length $\kappa^+$. In particular, any two $p$-points have the same Tukey class with respect to $\supseteq^*$.
\end{lemma}
\begin{remark}
    In the previous model, where the spectrum of a cardinal is large,  there are different Tukey classes of ultrafilters with respect to $\supseteq^*$.
\end{remark}
\begin{proposition}
    If $U$ is a $p$-point then $(U,\supseteq^*)<_T (U,\supseteq)$.
\end{proposition}
\begin{proof}
    Indeed, the identity map witnesses that $\leq_T$, and equality cannot hold since $(U,\supseteq^*)$ is $\kappa^+$-directed (as it is $p$-point) while $(U,\supseteq)$ is not $\kappa^+$-directed.
\end{proof}
\begin{definition}\label{def: Gal star}
    Let $U$ be an ultrafilter. We say that $\Gal^*(U,\mu,\lambda)$ for every sequence $\l A_i\mid i<\lambda\r\in [U]^\lambda$ there exists $I\in[\lambda]^\mu$, such that the subsequence $\l A_i\mid i\in I\r$ has a pseudointersection.
\end{definition}
It is not hard to see that $\Gal(U,\mu,\lambda)$ implies $\Gal^*(U,\mu,\lambda)$. 
\begin{definition}[Kunen \cite{Kunen1993}] We say that an ultrafilter $U$ over $\kappa$
is $\lambda$-OK if for every $\l A_\alpha\mid \alpha<\kappa\r\subseteq U$, there exists $\l B_\beta\mid \beta<\lambda\r\subseteq U$ 
such that for all nonempty $\sigma\in[\lambda]^{<\kappa}$, we have
$\bigcap_{\beta\in\sigma }B_\beta\subseteq^* A_{|\sigma|}$.
\end{definition}
\begin{proposition}
    If $U$ is a non $p$-point $\kappa$-complete ultrafilter which is $2^\kappa$-OK, then $\neg\Gal^*(U,\kappa,2^\kappa)$.
\end{proposition}
\begin{proof}
    Take any sequence $\l A_i\mid i<\kappa\r$ which has no pseudo intersection. Since $U$ is $\kappa$-complete we may assume that $A_i$ is decreasing in $\subseteq$. This assumption guarantees that there is no pseudo intersection for $\l A_i\mid i\in I\r$ where $I\in[\kappa]^\kappa$. By $2^\kappa$-OK, there is a sequence $\l B_i\mid i<2^\kappa\r$ such that for every $a\in[2^\kappa]^{<\kappa}$, $\cap_{i\in a}B_i\subseteq^* A_{|a|}$. It follows that $\l B_i\mid i<2^\kappa\r$ witnesses $\neg\Gal(U,\kappa,2^\kappa)$. 
\end{proof}
Kunen proved \cite{Kunen1993} that on $\omega$ there are non-$p$-points which are $2^\omega$-OK.  In the first study of the Tukey order on ultrafilters, Milovich \cite{Milovich08} proved that if $U$ is a non-$p$-point which is $2^\omega$-OK, then $U$ is Tukey top. Similar to Theorem \ref{TheNonGlavintheorem}, we can characterize using $\Gal^*(U,\kappa,2^\kappa)$ the ultrafilters which are Tukey top with respect to $\supseteq^*$.
 This improves Milovich's result (both for $\omega$ and $\kappa$).  

 \begin{theorem}\label{thm.3.16}
    Let $\kappa$ be either $\omega$ or a measurable cardinal. For any $\kappa$-complete ultrafilter $U$ over $\kappa$, $(U,\supseteq^*)\equiv_T([2^\kappa]^{<\kappa},\subseteq)$ if and only if $\neg\Gal^*(U,\kappa,2^\kappa)$. 
\end{theorem}
\begin{proof}
    Suppose $\l B_\alpha\mid \alpha<2^\kappa\r\in [U]^{2^\kappa}$ witnesses that $\neg\Gal^*(U,\kappa,2^\kappa)$, and let $g:[2^\kappa]^{<\kappa}\rightarrow \{B_\alpha\mid \alpha<2^\kappa\}$ be any enumeration. Let us prove that $g$ is unbounded. Suppose that $\mathcal{A}\subseteq [2^\kappa]^{<\kappa}$ is unbounded. Then $|\mathcal{A}|\geq \kappa$, so $g[\mathcal{A}]$ is a subset of $\{B_\alpha\mid \alpha<2^\kappa\}$ of cardinality $\kappa$ and therefore has no pseudointersection; hence $g[\mathcal{A}]$ is unbounded. In the other direction, suppose that $\Gal^*(U,\kappa,2^\kappa)$ holds, and let $g:[2^\kappa]^{<\kappa}\rightarrow U$.  We will prove that $g$ is not unbounded. If $|\Image(g)|<2^\kappa$, then there are $\kappa$-many elements mapped to the same set in $U$. Any $\kappa$-many elements in  $[2^\kappa]^{<\kappa}$ are unbounded  (since $\kappa$ is a strong limit) so $g$ is not unbounded. If $|\Image(g)|=2^\kappa$, there is $Y\subseteq\Image(g)$, such that $|Y|=\kappa$ which has a pseudointersection. Again $g^{-1}[Y]$ is unbounded in $[2^\kappa]^{<\kappa}$ but it is mapped to a bounded set.
\end{proof}
Milovich observed that a small modification in Isbell's construction gives an ultrafilter $U$ on $\omega$ such that $\neg\Gal^*(U,\omega,2^\omega)$. This modification, combined with the techniques of the previous section give the following:
\begin{theorem}\label{thm.3.17}
    Suppose that $\kappa$ is $\kappa$-compact. Then there is a $q$-point ultrafilter $U$ such that $\neg \Gal^*(U,\kappa,2^\kappa)$.
\end{theorem}

\begin{question}
    It is consistent from just a measurable cardinal that there is a $\kappa$-complete ultrafilter such that $\neg \Gal^*(U,\kappa,2^\kappa)$?
\end{question}
\begin{question}
    Is it consistent that $U$ is a $\kappa$-complete ultrafilter over $\kappa$ such that $(U,\supseteq^*)$ is not Tukey-top but $(U,\supseteq)$ is? Namely, is it consistent that  $\neg\Gal(U,\kappa,2^\kappa)$ but $\Gal^*(U,\kappa,2^\kappa)$?
\end{question}

\section{Continuous cofinal maps with special properties}

Many of the orders on ultrafilters over measurable cardinals such as the Rudin-Keisler order, the Mitchell order, the Ketonen order, and the Rudin-Frolik order are expressible in terms of the ultrapower $j_U:V\rightarrow M_U$ of a given $\kappa$-complete ultrafilter $U$. For example,  $U\leq_{RK} W$ if and only if there is an elementary embedding 
$k:M_U\rightarrow M_W$ such that $k\circ j_U=j_W$. One important property which allows  this characterization using ultrapowers is the fact that a Rudin-Keisler projection is  just a function $f:\kappa\rightarrow \kappa$ which is seen by the ultrapower of every $\kappa$-complete ultrafilter. A major difficulty when analyzing the Tukey order on $\kappa$-complete ultrafilters is that the projection maps (namely cofinal maps) are in general  functions $f:2^\kappa\rightarrow 2^\kappa$, which are extremely large objects
when compared to the closure degree of the ultrapower. However, in some situations, the information needed to compute a cofinal map can be coded into a function from a $\kappa$-sized set to a $\kappa$-sized set. This will allow us to work with cofinal maps in the ultrapower and derive some structural corollaries.


In the following, assume $\kappa$ is measurable and consider  $2^{\kappa}$ as  a generalized Baire space with the topology generated by basic open sets of the form
$\{x\in 2^{\kappa}: s\sqsubset x\}$, where $s\in 2^{<\kappa}$.
We use $U\re X$ to denote the set  $\{Y\in U : Y\sse X\}$.

The next theorem extends Theorem 20 of Dobrinen and Todorcevic in  \cite{Dobrinen/Todorcevic11}  for p-points over  $\om$ to p-points  over  measurable cardinals.
We show that any cofinal map from a p-point $U$ to a uniform ultrafilter on $\kappa$ is actually continuous upon restricting below some  suitable member of $U$.
Our proof follows  the  general form of that in \cite{Dobrinen/Todorcevic11},
but with some streamlining  modifications utilizing the measurability of $\kappa$.
For normal ultrafilters (Corollary \ref{cor.4.2}), we obtain  additional useful properties.

\begin{theorem}\label{Thm:DT20} 
Let $\kappa$ be measurable, and 
suppose $ U $ is a p-point on $\kappa$  and  $ V $ is a uniform ultrafilter on $\kappa$ such that $ U \ge_T V$.
Then 
for each monotone cofinal map $f: U \ra V  $,
there is an $X\in U $ such that the restriction of  $f$ to  $ U \re X$ is  continuous and has image which is cofinal in $ V $.

More precisely, let $\pi:\kappa\rightarrow\kappa$ be the minimal non-constant function mod $U$ (i.e.\ $[\pi]_U=\kappa$), and $\rho$ be the function $\rho_\alpha=\sup(\pi^{-1}[\alpha+1])+1$. Then for any strictly increasing sequence $\langle  \delta_{\al}\mid \al<\kappa\rangle$,
we can find a set $X=X(\l \delta_\alpha\mid \alpha<\kappa\r)$ such that for every $\alpha<\kappa$ and $A\in U\restriction X$, $f(A)\cap\delta_\alpha$ depends only on $A\cap \rho_\alpha$. Moreover,  there is a monotone function $\hat{f}:[\kappa]^{<\kappa}\rightarrow[\kappa]^{<\kappa}$ such that 
for each $A\in U\re X$, there are club many $\al<\kappa$ 
such that
\[
\hat{f}(A\cap \rho_\al)=
f(A)\cap \delta_{\al}
\]
and in particular
\[f(A)=\bigcup_{\al<\kappa}\hat{f}(A\cap \rho_\al )\]
Then defining   $g:P(\kappa)\rightarrow P(\kappa)$
by 
\[
g(A)=\bigcup_{\al<\kappa}\hat{f}(A\cap X\cap\rho_\al)
\]
we see that
\begin{enumerate}
    \item 
$g$ is a continuous monotone map, \item
$g\re U: U\rightarrow V$ is a cofinal map,
\item 
$g\re (U\re X)=f\re(U\re X)$.
\end{enumerate}
\end{theorem}

\begin{proof}
Let $U$ be a p-point on $\kappa$.
Suppose $V$ is a uniform ultrafilter such that $U \ge_T V  $, and 
 let 
 $f: U \ra V  $ be a
monotone cofinal map and  $\langle  \delta_{\al}\mid \al<\kappa\rangle$ be a strictly increasing sequence.
Construct a $\subseteq$-decreasing  sequence $\langle X_{\al}:\al<\kappa\rangle$ (i.e. $\alpha\leq\beta\Rightarrow X_\beta\subseteq X_\al$) of members of $ U $ such that for each $\al<\kappa$, $X_{\al}\cap \rho_\alpha=\emptyset$ and the following holds:
\begin{enumerate}
    \item[$(*)$]
    For each pair $s\sse \rho_\al$ and  $\delta\le\delta_\al$, if there is a $Y\in U $ such that $s=Y\cap\rho_\al$ and 
    $\delta\not\in f(Y)$, 
    then $\delta\not\in f(s\cup X_\al)$ (i.e. $f(s\cup X_\alpha)\cap \delta_\alpha\subseteq f(Y)\cap\delta_\alpha$).
\end{enumerate}

Take $X_0\in U $ such that 
for each $\delta\le \delta_0$, 
 $\delta\not\in f(X_0)$.
Such an $X_0$ exists
 since $ V  $ is uniform.
Suppose $1\le \al<\kappa$ and $\langle X_{\beta}\mid\beta<\al\rangle$ have been chosen satisfying $(*)$.
For each $s\sse\rho_\al$ and $\delta\le\delta_\al$, 
if there  is a $Y\in U$ such that 
$s=Y\cap \rho_\al$ and 
$\delta\not\in f(Y)$, then take $Y_{s,\delta}$ so that 
$Y_{s,\delta}\cap \rho_\al=s$ and $\delta\not \in f(Y_{s,\delta})$; otherwise, let $Y_{s,\delta}=[\rho_\al,\kappa)$.
Define 
\begin{equation}
X_\al=\bigcap\{Y_{s,\delta}:s\sse\rho_\al,\ \delta\le\delta_\al\} \cap \bigcap\{X_{\beta}:\beta<\al\}.
\end{equation}
Then $X_{\al}\cap\rho_\al=\emptyset$ and our construction  ensures that    $X_{\al}$ satisfies property $(*)$:
Given $s\sse\rho_\al$ and $\delta\le \delta_\al$,  if there is a $Y\in U $ such that $Y\cap \rho_\al=s$ and $\delta\not\in f(Y)$ then $\delta\not \in f(Y_{s,\delta})$
so by monotonicity of $f$, $\delta\not\in f(s\cup X_{\al})$.

 By definition of the $\rho_\alpha$'s  
 and the fact that $U$ is a p-point, there is an $X\in U$ such that
$X\setminus \rho_\alpha\subseteq X_\alpha$, for each $\al<\kappa$.
 We claim that $f$ is continuous when restricted to $ U \re X$.

Let $W\in U\re X$ and $\alpha<\kappa$. 
 It follows that 
 \begin{equation}
   W\setminus \rho_{\alpha}
   \sse X\setminus \rho_{\alpha}
   \sse X_{\alpha}
 \end{equation}
Let $s=W\cap\rho_\al$,  and  let  $\delta\le \delta_\al$ be given.
If $\delta\not\in f(W)$, then
by $(*)$ it follows that  $\delta\not \in f(s\cup X_{\al})$.
Since $X\setminus \rho_\al\sse X_{\al}$, 
by monotonicity also $\delta\not\in f(s\cup (X\setminus\rho_\al))$.
On the other hand, if  $\delta\in f(W)$, then  by monotonicity 
$\delta\in f(s\cup X\setminus \rho_\al)$.
Thus, 
\begin{equation}
f(W)\cap \delta_\al= f((W\cap\rho_\al)\cup (X\setminus \rho_\al))\cap \delta_\al
\end{equation}

Define the function $\hat{f}:[\kappa]^{<\kappa}\rightarrow[\kappa]^{<\kappa}$  as follows:  For $s\in[\kappa]^{<\kappa}$, letting 
$\alpha$ be the least such that 
$s\subseteq\rho_{\alpha}$, define
\begin{equation}
    \hat{f}(s)= f(s\cup (X\setminus \rho_\al))\cap \delta_\al
\end{equation}
This $\hat{f}$ is monotone and  uniformly continuous 
and for each $A\in U\re X$,
\begin{equation}
f(A)=\bigcup_{\al<\kappa}\hat{f}(A\cap\rho_\al)
\end{equation}
since there are club many $\al$, namely those in $\lim(A)\cap C_{\alpha\mapsto\rho_\alpha}$,
for which $f(A)\cap\delta_{\al}=\hat{f}(A\cap\rho_{\al})$.

Defining $g$  as in the statement of the theorem, it is routine to check that $g$ is as claimed.  
\end{proof}
\begin{corollary}\label{cor.4.2}
    If $U$ is normal, then  we have that $\rho_\alpha=\alpha+1$ and thus $f(W)\cap\delta_\alpha=f((W\cap\alpha+1)\cup (X\setminus \alpha+1))\cap\delta_\alpha$, for each $\al<\kappa$.
\end{corollary}
\begin{example}
Suppose that $U$ is a normal ultrafilter and $h_*(U)=V$ for some $h:\kappa\rightarrow\kappa$ (so in particular, $h$ can be taken to be one-to-one and $h(\alpha)>\alpha$).
Let 
 $f:U\rightarrow V$ be  defined by $f(X)=h''X$, 
 noting that $f$ is a
monotone cofinal map.
 Pick any sequence $\l \delta_\al>h(\alpha)\mid\alpha<\kappa\r$. 
 In the proof of Theorem \ref{Thm:DT20},
 for any $s\subseteq \alpha+1$ and $\delta<\delta_\alpha$, if $\delta\in h''s$ then  $Y_{s,\delta}=[\delta_\alpha,\kappa)$; otherwise, if $\delta\notin h''s$ then $Y_{s,\delta}=s\cup[\delta_\alpha,\kappa)$.
 Now $X_\alpha=[\delta_\alpha,\kappa)$ and   $$X=\Delta_{\alpha<\kappa}X_\alpha=\{\beta<\kappa\mid\forall \alpha<\beta, 
\ \delta_\alpha\leq\beta\}$$
Taking any $W\subseteq X$ and letting $s=W\cap\alpha+1$, we have that $W\cap\alpha+1$ decides that
$f(W)\cap\delta_\alpha= h''(W\cap\alpha+1)$.
Since any point $h(\beta)\in h''(W\setminus\alpha+1)$ will satisfy that $h(\alpha)<\delta_\alpha\leq\beta$, it follows 
 that $h'' W\cap (\alpha,\delta_\alpha)=\emptyset$. 
\end{example}

We point out that $\al+1$ is the best possible in the previous corollary.
For if the function $f(W)\cap\delta_\alpha$ depends only on $W\cap \alpha$, then  $W\cap\alpha$ decides $f(W)\cap \alpha+1$.  But this cannot be true in general since for example if $f$ is the identity function as above, then this would mean we can decide $W\cap\alpha+1$ using $W\cap\alpha$ below a certain set $X\in U$. But take any $\alpha\in X$, and take $W_1,W_2\subseteq X$ such that $X\cap\alpha=W_1\cap\alpha=W_2\cap\alpha$, $\alpha\in W_1$ and $\alpha\notin W_2$; then we get different decisions. 

If $U\neq V$, we can find $X\in U$ and $Y\subseteq X$ such that $\kappa\setminus X\in V$  and  $f(Y)\subseteq \kappa\setminus X\subseteq \kappa\setminus Y$. In this situation, we know that whenever $\alpha\notin W\subseteq Y$, then $W\cap\alpha=W\cap\alpha+1$ and we can decide $f(W)\cap \delta_\alpha$, and if $\alpha\in W$ then $\alpha\notin f(W)$.

 Raghavan showed in Corollary 11 of \cite{Raghavan/Todorcevic12} that there are no continuous cofinal maps from a Fubini product of two p-points $U,W$ on $\om$ to $U$, and the proof generalizes easily to ultrafilters on any countable base set.
On the other hand, the second author showed in  Theorem 4.4 of \cite{DobrinenFund20}
that Fubini products, even countable Fubini iterates, of p-points on $\om$ have cofinal maps which are canonized by finitary maps which behave similarly  to continuous maps, but the proof is rather long.
Theorem \ref{ProductCor} in the next section will show that, surprisingly, for $\kappa$-complete ultrafilters, Fubini products and cartesian products 
have the same Tukey type.

We now extend Theorem \ref{Thm:DT20}  to cartesian products of finitely many p-points.
This theorem, along with 
 Theorem \ref{ProductCor},
will aid  in understanding the cofinal types of 
Fubini products of finitely many $\kappa$-complete ultrafilters.

\begin{theorem}\label{Prodthm.20DT-kappa}
Suppose $U_1,\dots,U_n$ are $p$-points over $\kappa>\omega$, $\pi_1,\dots,\pi_n$ are functions such that $[\pi_i]_{U_i}=\kappa$, and  $V$ is a uniform ultrafilter on $\kappa$ such that $U_1\times\dots\times U_n\ge_TV$.
Then 
for each monotone cofinal map $f:U_1\times\dots \times U_n\ra V$ and every sequence $\l \delta_\alpha\mid \alpha<\kappa\rangle$,
there are sets $X_i\in U_i$ such that the restriction of  $f$ to  $U_1\times\dots\times U_n\re \l X_1,\dots,X_n\r$ is monotone, cofinal, and uniformly continuous, i.e.\ for each $\alpha<\kappa$, $f(\l B_1,\dots,B_n\r)\cap \delta_\al$ is determined by $\l B_1\cap\rho^1_\alpha,\dots,B_n\cap\rho^n_\alpha\r$, where for each $1\leq i\leq n$ and each $\alpha<\kappa$, $\rho^i_\alpha=\sup(\pi_i^{-1}[\alpha+1])+1$.
In particular,  there is a continuous monotone  cofinal map from $U_1\times\dots\times U_n$ to $V$.
\end{theorem}

\begin{proof}
As before, let us construct a $\supseteq$-decreasing sequence for each $1\leq i\leq n$, $\langle X_{\al,i}:\al<\kappa\rangle$  of members of $ U _i$ such that for all $\al<\kappa$, $X_{\al,i}\cap\rho^i_\al=\emptyset$ and the following holds:
\begin{enumerate}
    \item[$(*)$]
    For each pair $s_i\sse\rho^i_\al$ and  $\delta\le\delta_\al$, if there is a $\vec{Y}=\l Y_1,\dots, Y_n\r\in U_1\times\dots\times U_n$ such that $s_i=Y_i\cap\rho^i_\al$ and 
    $\delta\not\in f(\vec{Y})$, 
    then $\delta\not\in f(\l s_1\cup X_{\al,1},\dots,s_n\cup X_{\al,n}\r)$.
\end{enumerate}

Take $X_{0,i}\in U_i$ such that 
 $0\not\in f(\l X_{0,1},\dots,X_{0,n}\r)$.
 Since $V$ is uniform, and $f$ is cofinal, such an $X_{0,i}$ exists.
Suppose $1\le \al<\kappa$ and $\langle X_{\beta,i}:\beta<\al\rangle$ have been constructed. 
For each $s_i\sse\rho^i_\alpha$ and $\delta\le\delta_\al$, 
if there  are a sets $Y_i\in U _i$ such that $s_i=Y_i\cap \rho^i_\al$ and $\delta\not\in f(\l Y_1,\dots, Y_n\r)$, then let $Y_{s_1,..,s_n,i,\delta}$ be such a $Y_i$; otherwise, let  $Y_{s_1,..,s_n,i,\delta}=\kappa$.
Define 
\begin{equation}\label{eq}
X_{\al,i}=\bigcap\{Y_{s_1,\dots,s_n,i,\delta}:s_1\subseteq \rho^1_\alpha,\dots,s_n\sse\rho^n_\al,\ \delta\le\delta_\al\} \cap \bigcap\{X_{\beta,i}:\beta<\al\}.
\end{equation}
Note that $X_{\al,i}\cap\rho^i_\al=\emptyset$.
 Moreover,  given $s_i\sse\rho^i_\al$, 
 if there is a $Y_i\in U_i$ such that $Y_i\cap\rho^i_\al=s_i$ then $$f(\l s_1\cup X_{\al,1},\dots,s_n\cup X_{\al,n}\r)\cap \delta_\alpha\subseteq f(\l Y_1,\dots, Y_n\r)\cap \delta_\alpha,$$ by choice of $X_{\al,i}$ and monotonicity of $f$.
 
 Since each $U_i$ is a normal ultrafilter, we let $X_i^*:=\Delta^*_{i<\kappa}X_{\al,i}\in U_i$, then $X_i^*\setminus\rho^i_\alpha\subseteq X_{\alpha,i}$. 
 We claim that $f$ is continuous when restricted to $U_1\times\dots\times U_n\re \l X^*_1,\dots,X^*_n\r$.
 Toward this, let $W_i\in U_i$ be such that $W_i\sse X_i^*$, let  $\alpha<\kappa$ and let $s_i=W_i\cap \rho^i_\alpha$.
It follows that 
 \begin{equation}
   W_i\setminus \rho^i_\alpha
   \sse X^*_i\setminus \rho^i_\alpha
   \sse X_{\alpha,i}.
 \end{equation}
 Thus, 
     $$f(\l W_1,\dots, W_n\r)\cap\delta_\alpha\subseteq f(\l s_1\cup X^*_{1}\setminus\rho^1_\alpha,\dots,s_n\cup X^*_{n}\setminus\rho^n_\alpha)\cap \delta_\alpha\sse$$
     $$\sse f(s_1\cup X_{\alpha,1},\dots,s_n\cup X_{\alpha,n})\cap\delta_\alpha\sse f(W_1,\dots ,W_n)\cap\delta_\alpha$$
Then $f(\l W_1,\dots, W_n\r)\cap\delta_\alpha=f(\l s_1\cup X^*_{1}\setminus\rho^1_\alpha,\dots,s_n\cup X^*_{n}\setminus\rho^n_\alpha)\cap \delta_\alpha$ is continuous on $U_1\times\dots\times U_n\re \l X^*_1,\dots,X^*_n\r$.
\end{proof}

Let us derive similar corollaries to the ones from \cite[Corollaries 21\&23]{Dobrinen/Todorcevic11}.
First, since there are only $2^\kappa$-many continuous cofinal maps we get the following:
\begin{corollary}
    There are only $2^\kappa$ many ultrafilters which are Tukey below a $p$-point. Moreover, every $\leq_T$-chain of $p$-points has order-type at most $(2^\kappa)^+$.
\end{corollary}
Next, we apply the free set lemma of Hajnal \cite{juhasz1979cardinal}.
\begin{lemma}[The free set lemma of Hajnal] If $|X|=\theta$, $\lambda<\theta$, and $F:X\rightarrow P(X)$ satisfies $x\notin F(x)$ and $|F(x)| <\lambda$, for all $x\in X$, then there is a $Y \subseteq X$, $|Y|=\theta$ with
$x \notin F(y)$ and $y \notin F(x)$ for all $x, y \in Y$.
\end{lemma}
\begin{corollary}\label{ManyP-pointsCor}
If $\chi$ is a set of 
 at least $(2^\kappa)^+$-many $p$-point ultrafilters on $\kappa$, then there is $Z\subseteq \chi$
 of cardinality $|\chi|$
such that every distinct $U,V\in Z$  are incomparable in the Tukey order.    
\end{corollary}
\begin{remark}
    In the Kunen-Paris model there are $2^{(\kappa^+)}$-many  normal ultrafilters, $2^\kappa=\kappa^+$ and $2^{(\kappa^+)}$ can be made arbitrarily large. In particular, there are $2^{(2^\kappa)}$-many incomparable normal ultrafilters in the Tukey order.  
\end{remark}

\section{On the Tukey class of finite powers of ultrafilters}
\rm

In this section, we analyze the Tukey
structure
of (cartesian) products  and  Fubini products and limits of ultrafilters. 
some of these arguments are adaptations to uncountable cardinals of theorems
from Section 4 of \cite{Dobrinen/Todorcevic11}.  However, 
 it turns out that the situation for measurable cardinals is very different from  that of  ultrafilters on $\omega$.

 Let us start with the following:
\begin{lemma}
    \label{thm.DT32}
For any ultrafilters $U, V, U_\alpha$ $(\alpha < \kappa)$, we have that:
\begin{enumerate}
    \item $\sum_{U} U_\alpha \leq_T U \times \prod_{\alpha<\kappa}U_\alpha$.
    \item $U\cdot{V}\leq_T U\times \prod_{\alpha<\kappa}V$.
    \item $U \cdot U \leq_T
\prod_{\alpha<\kappa}U$.
\end{enumerate}
\end{lemma}


\begin{proof}
The proof of $(1)$ is  a straightforward generalization of \cite[Theorem 32]{Dobrinen/Todorcevic11}: Given any $A\in \sum_{U}U_\alpha$, define $(A)_\alpha=\{\beta<\kappa\mid \l\alpha,\beta\r\in A\}$
and let $\mathcal{B}$ be the set of those $A\in \sum_{U} U_\alpha$ such that for each $\al<\kappa$, either $(A)_\al\in U_\al$ or else $(A)_\al=\emptyset$.
Then $\mathcal{B}$ is a filter base for $\sum_{ U} U_\alpha$.
Given $A\in\mathcal{B}$ define
$\pi_0(A)$ to be the set of $\al<\kappa$ such that $(A)_\al\in U_\al$,
and for each $\al<\kappa$, define $u_\al(A)=(A)_\al$ if $\al\in\pi_0(A)$ and $\kappa$ otherwise.
Define
$f(A)$ to be $\langle\pi_0(A)\rangle^{\frown}\langle u_\al(A):\al<\kappa\rangle$.
Then $f$ is a Tukey map. $(2)$ and $(3)$ are immediate from $(1)$.
\end{proof}
Let $U$ be a $\sigma$-complete ultrafilter over $\kappa$, $\pi:\kappa\rightarrow \kappa$ be the least non-constant  function mod $U$, and let $\l A_i\mid i<\kappa\r$ be any sequence of sets in $U$, we define the \textit{modified diagonal intersection} of the sequence by
$$\Delta^*_{i<\kappa}A_i:=\{\alpha<\kappa\mid\forall i<\pi(\alpha), \ \alpha\in A_i\}$$

\begin{fact}\label{fact.5.2}
    For any sequence $\l A_i\mid i<\kappa\r\in U$, where $U$ is $\sigma$-complete, $\Delta^*_{i<\kappa}A_i\in U$. 
\end{fact}
\begin{proof}
    Just otherwise, the complement $B:=\kappa\setminus \Delta^*_{i<\kappa}A_i$ would be in $U$, and for every $\beta\in B$, there is $i_\beta<\pi(\beta)$ such that $\beta\notin A_{i_\beta}$, but then $[\beta\mapsto i_\beta]_U<[\pi]_U$, and since $[\pi]_U$ is the least non-constant function, then there is a set $B'\in U\restriction B$ and $i^*$ such that for every $\beta\in B'$, $i_\beta=i^*$. This implies that $B'\cap A_{i^*}=\emptyset$, contradicting that both $B'$ and $A_{i^*}$ are in $U$.
\end{proof}
\begin{lemma}\label{NormalLemma}
Let $\omega<\kappa$ and $ U $ be a $p$-point over $\kappa$. Then there is a function $h:\kappa\rightarrow \kappa$ such that for every $\l A_i\mid i<\kappa\r\in\prod_{i<\kappa} U $,  and every $\alpha<\kappa$, $(\Delta^*_{i<\kappa}A_i)\setminus h(\alpha)\subseteq A_\alpha$.
\end{lemma}

\begin{proof}
Recall that $\pi:\kappa\rightarrow\kappa$ denotes the least non-constant function. 
Since $ U $ is a $p$-point, for each $\alpha<\kappa$, $\pi^{-1}[\alpha]:=\{\gamma<\kappa\mid \pi(\gamma)<\alpha\}$ is bounded in $\kappa$;
let $h(\alpha)=(\sup\pi^{-1}[\alpha+1])+1$. Let $\l A_i\mid i<\kappa\r\in\prod_{i<\kappa} U $ be given.
Then for any $\alpha<\kappa$, and any $\nu\in(\Delta^*_{i<\kappa}A_i)\setminus h(\alpha)$, $\nu\notin \pi^{-1}[\alpha+1]$ and therefore $\pi(\nu)>\alpha$ which implies by the definition of the modified diagonal intersection that $\nu\in A_\alpha$. 
\end{proof}

In general, if $U$ is not a $p$-point we cannot prove that there is such a uniform function $h$ which works for every sequence of sets, but instead we can do the following:
\begin{lemma}
    Let $\omega<\kappa$ and $U$ be any $\kappa$-complete ultrafilter over $\kappa$. Then there is a sequence of sets $\l B_\alpha\mid \alpha<\kappa\r$ such that $B_\alpha\notin U$ for each $\alpha$, such that for every sequence $\l A_i\mid i<\kappa\r\in \prod_{i<\kappa}U$, for every $\alpha<\kappa$, $(\Delta^*_{i<\kappa}A_\alpha)\setminus B_\alpha\subseteq A_\alpha$.
\end{lemma}
    \begin{proof}
        Let $\pi:\kappa\rightarrow\kappa$ be as before, and define $B_\alpha=\pi^{-1}[\alpha+1]$ for each $\al<\kappa$. Since $\pi$ is not constant mod $U$ and $U$ is $\kappa$-complete, each  $B_\alpha\notin U$. Let $\l A_i\mid i<\kappa\r\in\prod_{i<\kappa} U $ be any sequence.
To see that  $(\Delta^*_{i<\kappa}A_i)\setminus B_\alpha\subseteq A_\alpha$, let  $\nu\in (\Delta^*_{i<\kappa}A_i)\setminus\pi^{-1}[\alpha+1]$. Then $\pi(\nu)\geq\alpha+1>\alpha$, so $\nu\in A_\alpha$ by the definition of $\Delta^*_{i<\kappa}A_i$.
    \end{proof}

The following theorem is very different from the situation on $\omega$:
\begin{lemma}\label{nonTrivialdirection}
    Let $\omega<\kappa$ be measurable. For each $\kappa$-complete ultrafilter $U$ over $\kappa$, $ \prod_{i<\kappa}U\leq_T U$.
\end{lemma}
\begin{proof} Let us define the map $F:\prod_{i<\kappa}U\rightarrow U$ by
$$F(\l A_i\mid i<\kappa\r)=\Delta^*_{i<\kappa}A_i$$ To see that $F$ is unbounded, let $\{\l A^j_i\mid i<\kappa\r\mid j\in I\}$ be unbounded in $\prod_{i<\kappa}U$ and suppose toward contradiction that there is $B\in U$ such that $B\subseteq \Delta^*_{i<\kappa}A^j_i$ for every $j\in I$. Let $\l B_i\mid i<\kappa\r$ be the sequence from the previous lemma,
 and note that
$\l B\setminus B_i\mid i<\kappa\r\in \prod_{i<\kappa}U$. 
Also note that by the property of the sequence $\l B_i\mid i<\kappa\r$, for each $i<\kappa$ and each $j\in I$, $$B\setminus B_i\subseteq (\Delta^*_{i<\kappa}A^j_i)\setminus B_i\subseteq A^j_i.$$ We conclude that in the everywhere domination order of the product $\prod_{i<\kappa}U$, we have $\l B\setminus B_i\mid i<\kappa\r\geq \l A^j_i\mid i<\kappa\r$. Therefore, the sequence $\l B\setminus B_i\mid i<\kappa\r$ bounds the set $\{\l A^j_i\mid i<\kappa\r\mid j\in I\}$, contradicting our initial assumption that this set should be unbounded. 
\end{proof}

The next several corollaries demonstrate the striking difference between ultrafilters on $\om$ versus on a measurable cardinal.
For ultrafilters on $\om$, 
Dobrinen and Todorcevic 
in \cite{Dobrinen/Todorcevic11} showed 
that Fubini products of rapid p-points are Tukey equivalent to their cartesian products, whereas if rapidness is dropped, then the equivalence  fails.
Here, rapidness and even p-point can be dropped: the Tukey equivalence between finite Fubini products and cartesian products holds simply under the assumption that the ultrafilters are $\kappa$-complete.

\begin{theorem}
    \label{ProductCor} Let $\omega<\kappa$ be measurable.
    For every two $\kappa$-complete ultrafilters  $U,V$ over $\kappa$, $U\cdot V\equiv_T U\times V$. Moreover, let $U_1,\dots,U_n$ be any $\kappa$-complete ultrafilters over $\kappa$.
    Then $U_1\cdot{\dots}\cdot U_n\equiv_T\prod_{i=1}^n U_i$.\end{theorem}
\begin{proof}
    $U\times V\leq_{T} U\cdot V$ by Proposition \ref{productAndCart}. For the other direction, $U\cdot V\leq_T U\times \prod_{\alpha<\kappa}V$ by Lemma \ref{thm.DT32} item $(2)$, and by Lemma \ref{nonTrivialdirection}, $  U\times \prod_{\alpha<\kappa}V\leq_T U\times V$. The moreover part follows by a straightforward induction. 
\end{proof}

\begin{theorem}\label{PowerCor}
    For any $\kappa$-complete ultrafilter $U$ over a measurable cardinal $\kappa$, $U^n\equiv_T U$.
\end{theorem}
\begin{proof}
    $U^n\equiv_T U\times U\times\dots\times U\equiv_T U$.
\end{proof}
\begin{corollary}
    If $U$ is a Galvin ultrafilter over a measurable cardinal $\kappa$ (namely, $\Gal(U,\kappa,2^\kappa)$) then for every $n$, $U^n$ is Galvin.
\end{corollary}
\begin{proof}
    $U$ is Galvin iff $U$ is not Tukey-top, and $U^n$ is Galvin iff $U^n$ is not Tukey-top. Since $U\equiv_T U^n$, we see that $U$ is Galvin iff $U^n$ is Galvin. 
\end{proof}
\begin{question}
    Suppose that $U$, $V_\alpha$ are all Galvin. Does it follow that $\sum_UV_\alpha$ is  Galvin? Is the Fubini product of two Galvin ultrafilters always Galvin?
\end{question}

In the next section, we will provide a partial answer for this question when we restrict ourselves to a certain class of ultrafilters called {\em basically generated}.

\section{Basic and basically generated ultrafilters}\label{sec.bg}

One generalization from the theory on $\omega$ is the definition of the so-called \textit{basic} ultrafilters, which gives another characterization of $p$-point.
The following is the  notion of convergence on the generalized Baire  space $2^{\kappa}$, with the topology generated by basic open sets of the form $N_s=\{f\in 2^{\kappa} \mid s\sqsubset f\}$,
where $s\in 2^{<\kappa}$.
   Given $\l A_i\mid i<\kappa\r\subseteq P(\kappa)$, we write $\lim_{i\rightarrow \kappa}A_i=A$  to mean that for every $\delta<\kappa$ there is $i^*<\kappa$, such that for every $i\geq i^*$, $A_i\cap \delta= A\cap \delta$, in other words, the sequence $A_i$ converges to $A$.

The notion of a basic partial order first appeared in \cite{Solecki/Todorcevic04} in the context of separable metric spaces. 
The following extends the notion of basic ultrafilter on $\om$ in \cite{Dobrinen/Todorcevic11}  to uncountable cardinals. 

\begin{definition}
A $\kappa$-complete ultrafilter 
$U$ is called \textit{basic} if for every sequence $\l A_i\mid i<\kappa\r$ such that $\lim_{i\rightarrow\kappa} A_i=A\in U$ there is a subsequence $\l A_{i_\alpha}\mid \alpha<\kappa\r$ such that $\cap_{\alpha<\kappa}A_{i_\alpha}\in U$. 
\end{definition}

 Theorem 14 of \cite{Dobrinen/Todorcevic11} showed that for ultrafilters on $\omega$, p-points are exactly basic ultrafilters. 
We now prove analogous results for ultrafilters on uncountable cardinals.
First we present a  slightly different definition of basic which will turn out to be equivalent:
\begin{definition}
    A $\kappa$-complete ultrafilter 
$U$ is called {\em uniformly basic} if for every sequence $\l A_i\mid i<\kappa\r$ such that $\lim_{i\rightarrow\kappa} A_i=A\in U$ there is a function $f:\kappa\rightarrow\kappa$ such that whenever \textbf{$f\leq_{bd}g$}, $\cap_{\alpha<\kappa}A_{g(\alpha)}\in U$. 
\end{definition}
Clearly uniformly basic implies basic. 
\begin{proposition}\label{Prop:ppt implies basic}
$p$-point implies uniformly basic.
\end{proposition}

\begin{proof}
Suppose that $U$ is a $p$-point and let $\pi$ be the least non-constant function mod $U$. Since $U$ is a $p$-point, there is $Y\in U$ such that for every $\gamma<\kappa$, $\pi^{-1}[\gamma]\cap Y$ is bounded in $\kappa$. For each $\alpha<\kappa$, pick $\rho_\alpha:=\sup(\pi^{-1}[\alpha+1]\cap Y)+1$. To see that $U$ is uniformly basic, let $\l A_i\mid i<\kappa\r$ be a sequence converging to some set $A\in U$. 
By definition of convergence, there is a sequence $\delta_\alpha$ such that for any $i\geq\delta_\alpha$, $A_i\cap \rho_{\alpha}=A\cap\rho_{\alpha}$;
define $f(\al)=\rho_{\al}$.
Let us prove that for any $f\le_{bd} g$ we have that $\cap_{\alpha<\kappa}A_{g(\alpha)}\in U$.  
Towards this, let $\theta<\kappa$ be such that for every $\theta\leq\alpha<\kappa$, $\delta_\alpha\leq g(\alpha)$.
It suffices to prove that $(\cap_{\beta<\theta}A_{g(\beta)})\cap A\cap Y\cap \Delta^*_{\alpha<\kappa}A_{\delta_\alpha}\subseteq \cap_{\alpha<\kappa}A_{g(\alpha)}$,
where\footnote{To see that $\Delta^*_{\alpha<\kappa}A_{\delta_\alpha}\in U$, suppose  not; then for a set of $\nu$'s in $U$, there is $\alpha_\nu\in \pi(\nu)$ such that $\nu\notin A_{\delta_\alpha}$. Since $\pi$ is minimal, then $\nu\mapsto\alpha_\nu$ is constant mod $U$, say on a set $Z$ with value $\alpha^*$. Then for every $\nu\in Z$, $\nu\notin A_{\delta_{\alpha^*}}$, contradiction.}
$$\Delta^*_{\alpha<\kappa}A_{\delta_\alpha}:=\{\nu<\kappa\mid \forall \alpha<\pi(\nu),\ \nu\in A_{\delta_{\alpha}}\}$$
Indeed, suppose $\xi\in (\cap_{\beta<\theta}A_{g(\beta)})\cap A\cap Y\cap \Delta^*_{\alpha<\kappa}A_{\delta_\alpha}$.  
Let $\alpha_0<\kappa$ and split into cases:
\begin{enumerate}
    \item If $\alpha_0<\theta$, then 
   clearly  $\xi\in A_{g(\alpha_0)}$.
    \item If $\alpha_0<\pi(\xi)$, then by the definition of $\Delta^*_{\alpha<\kappa}A_{g(\alpha)}$, $\xi\in A_{g(\alpha_0)}$.
    \item 
    If $\alpha_0\geq \max(\pi(\xi),\theta)$ then $\xi\in\pi^{-1}[\alpha_0+1]\cap Y$,
which means that $\xi<\rho_{\alpha_0}$, so $\xi\in A\cap \rho_{\alpha_0}$. Also we conclude that $\delta_{\alpha_0}\leq g(\alpha_0)$, and by the choice of $\delta_{\alpha_0}$, 
it follows that $A_{g(\alpha_0)}\cap\rho_{\alpha_0}=A\cap\rho_{\alpha_0}$.
Thus $\xi\in A_{g(\alpha_0)}$.
\end{enumerate}

\end{proof}



The following proposition implies that basic, uniformly basic, and $p$-points are all equivalent:
\begin{proposition}\label{Prop:Basicimpliesppt}
Basic implies $p$-point.
\end{proposition}

\begin{proof}
Suppose that $U$ is basic and uniform. Let $\l X_i\mid i<\kappa\r\in [U]^\kappa$ be a sequence of sets in $U$. We
want to prove that this sequence
has a pseudointersection. By $\kappa$-completeness we may assume that the sequence of  $X_i$
is decreasing. 
Define the sets $Y_i=X_i\cup i$. Clearly, $Y_i\supseteq X_i$ and $\lim_{i\rightarrow\kappa}Y_i=\kappa$. Hence, since $U$ is basic, there is a subsequence $\l i_\alpha\mid \alpha<\kappa\r$ such that $X^*:=\cap_{\alpha<\kappa} Y_{i_\alpha}\in U$. Note that for each $i<\kappa$, there is $\alpha<\kappa$ such that $i<i_\alpha$ and therefore $X_i\supseteq X_{i_\alpha}=^*Y_{i_\alpha}\supseteq X^*$.
Hence $U$ is a
$p$-point.
\end{proof}
  \begin{proposition}
  Basic implies Galvin.
  \end{proposition}
  Although this proposition follows from Theorem \ref{thm.bgnottop}, we will provide a short proof for this proposition separately as we believe that it may be of interest to the reader who is familiar with the proof of Galvin's Theorem for normal filters, to see the similarities with the parallel theory for ultrafilters on $\omega$.
\begin{proof}
    Let $\l A_i\mid i<\kappa^+\r$ be any sequence of sets in $U$,   and for each $\alpha<\kappa$, $i<\kappa^+$ define
    $$H_{\alpha,i}:=\{j<\kappa^+\mid A_i\cap\alpha=A_j\cap\alpha\}$$
    Galvin observed that:
    \begin{lemma}
       There is  an $i^*$ such that for every $\alpha<\kappa$, $|H_{i^*,\alpha}|=\kappa^+$.
    \end{lemma}
    The proof can be found in \cite[Proposition 5.13]{Parttwo}. Then we can find a sequence of distinct $i_\alpha$'s such that $A_{i_\alpha}\cap\alpha=A_{i^*}\cap\alpha$. In particular, $\text{lim}_{\al\rightarrow\kappa} A_{i_\alpha}=A_{i^*}\in U $, and since $U$ is basic, there is a  subsequence $\l \al_j \mid j<\kappa\r$ such that $\cap_{j<\kappa} A_{i_{\al_j}}\in U$.
\end{proof}

The next  notion of 
basically generated ultrafilter   on $\om$ first appeared in  \cite{Dobrinen/Todorcevic11}
as a means to abstract 
 properties of  p-points responsible for  being  strictly below the Tukey top.
 In \cite{Dobrinen/Todorcevic11}, Dobrinen and Todorcevic 
proved that the collection of basically generated ultrafilters
on $\om$  contains all p-points and is closed under countable iterates of Fubini limits, and that all basically generated ultrafilters are not Tukey-top \cite[Theorems  16, 18]{Dobrinen/Todorcevic11}.
 We prove the analogues for $\kappa$-complete ultrafilters on $\kappa$ 
 (Theorems \ref{thm.bgnottop} and \ref{thm.bgclosedclass})
 building on those proofs but then
utilizing  some different techniques.
This is another improvement of Galvin's  Theorem,
and the theorem due to the first author from \cite{SatInCan} which asserts that a $p$-point sum of $p$-points has the Galvin property.

\begin{definition}\label{defn.bg}
An ultrafilter $U$ on $\kappa$  is called {\em basically generated} if it
is uniform, $\kappa$-complete, and  has a $\kappa$-complete filter base\footnote{Namely, $B$ is cofinal in $(U,\supseteq)$, and closed under intersecting less than $\kappa$-many of its elements.} $B\sse U$  with the property that to each sequence $\{A_i \mid i<\kappa\}\sse B$ converging to an element of $B$,
there corresponds
a  function $f$ such that for every  $f\le_{bd} g$,  we have $\cap_{\al<\kappa} A_{g(\al)}\in U$.
\end{definition}
Note that this definition differs from the one in \cite{Dobrinen/Todorcevic11}. 
\begin{theorem}\label{thm.bgnottop}
If  $U$ is a basically generated ultrafilter on $\kappa$, then $U$ is not Tukey-top.
\end{theorem}

\begin{proof}
  Let $B\sse U$ be a filter base witnessing that $U$ is basically generated.  
Fix a subset
$C\sse U$ of cardinality $2^{\kappa}$; for each $X\in C$ fix  one $Y_X\in B$ such that $Y_X\sse X$.
If there is a  subset $D\sse C$ of cardinality $\kappa$ 
for which all $Y_X$, $X\in D$, are equal, then 
there is a subset of $C$ of size $\kappa$ with intersection  in $U$.

Otherwise, $|\{Y_X  \mid X\in C\}|\ge  \kappa^+$.
Without loss of generality, assume that for each $X\ne X'$ in $C$, $Y_X\ne Y_{X'}$.
For each $s\in 2^{<\kappa}$,
let $C_s=\{X\in C \mid Y_X\cap \dom(s)= s^{-1}[\{1\}]\}$.
Note that for each 
 $\al<\kappa$, there is at least one $s\in 2^\al$ for which $|C_s|\ge \kappa^+$.

\begin{claim}
    There is an $X\in C$ such that for every $\alpha<\kappa$, $|C_{s_\alpha}|\geq\kappa^+$, where $s_\alpha$ is the characteristic function of $Y_X\cap\alpha$.
\end{claim}
\begin{proof}[Proof of claim.]
    Suppose otherwise. Then for each $X\in C$, there is a minimal $\alpha_X<\kappa$ such that $C_{s_X}$ has size at most $\kappa$, where $s_X$ is the characteristic function of $Y_X\cap\alpha_X$. Let $F=\{s_X\mid X\in C\}$. Then $F\subseteq 2^{<\kappa}$ so $|F|\leq\kappa$.
But then $C=\bigcup_{s\in F} C_s$ has size  at most $\kappa$, a contradiction.  
\end{proof}

Thus, fix  $X\in C$ and the sequence $\langle s_{\al} \mid \al<\kappa\rangle$ from the claim such that 
each $|C_{s_\al}|\ge\kappa^+$.
Now 
we build a convergent subsequence of $\{Y_X \mid X\in C\}$ quite simply:
For $\al<\kappa$, take any 
$X_\al\in C_{s_\al}$.
Then $\langle Y_{X_\al} \mid\al<\kappa\rangle$ is a sequence of members of $B$ which converges to $Y_X$ which is a member of $B$.
Since $B$ witnesses that $U$ is basically generated,
 there is a subsequence 
 $\langle \al_i \mid i<\kappa\rangle$
 such that $Y^*:=\bigcap_{i<\kappa} Y_{X_{\al_i}}$ is in $U$.
Finally, 
let $D=\{X\in C \mid Y^*\subseteq X\}$, and note that $\{X_{\alpha_i}\mid i<\kappa\}\subseteq D$; in particular, $|D|\ge \kappa$.
Then $\bigcap D\contains Y^*$ which is in $U$.  
Hence, $U$ is not Tukey-top.
\end{proof}

\begin{theorem}\label{thm.bgclosedclass}
    Suppose that $U$ and
    $V_{\al}$, $\al<\kappa$, are basically generated ultrafilters on $\kappa$. 
    Then $W:=\sum_U V_\al$ is basically generated (with respect to the product topology on $\kappa\times \kappa$). 
\end{theorem}

\begin{proof}
    Let $B$ and $(B_\alpha)_{\alpha<\kappa}$  be $\kappa$-complete filter bases for $U,(V_\alpha)_{\alpha<\kappa}$ witnessing that $U,(V_\alpha)_{\alpha<\kappa}$ are basically generated, respectively. Let $\pi_0:\kappa\times\kappa\rightarrow\kappa$ be the projection map to the first coordinate:  $\pi_0(\al,\beta)=\al$.
    For $X\sse\kappa\times\kappa$ and $\al<\kappa$, let $(X)_{\al}$ denote $\{\beta<\kappa \mid (\al,\beta)\in X\}$, the $\al$-th fiber of $X$.
  Define 
\[
    C=\{X\in W \mid \pi_0[X]\in B \wedge \forall \al<\kappa\, ((X)_{\al}=\emptyset \vee (X)_\al\in B_\al)\}
    \]
Then $C$ is a filter base for $W$ which is $\kappa$-complete.

Consider a converging sequence $X_\al\rightarrow X$ in $C$.
Note that $\pi_0[X_\al]$ converges to some member of  $ U$ (possibly not in $B$) containing $\pi_0[X]$.
Since 
$\pi_0[X]$ is  in $B$ and since $X\in C$,
 we make the following modification:
Let $Y_\al= X_\al\cap (\pi_0[X]\times\kappa)$, note that $Y_\al$ is in $C$ as $\pi_0[Y_\alpha]=\pi_0[X_\alpha]\cap\pi_0[X]\in B$, and for every $\xi<\kappa$, if $\xi\notin \pi_0[Y_\alpha]$ then $(Y_\alpha)_\xi=\emptyset$ and if $\xi\in \pi_0[Y_\alpha]$, then $(Y_\alpha)_\xi=(X_\alpha)_\xi$. 
 Let
 $A_{\al}$ denote $\pi_0[Y_{\al}]$ and $A$ denote $\pi_0[X]$.
Then we have that
$\lim_{\al\rightarrow\kappa}A_{\al}=A$,
$\lim_{\al\rightarrow\kappa}Y_\al= X$,
and for each $\xi<\kappa$, 
 $\lim_{\al\rightarrow\kappa}(Y_\alpha)_\xi=(X)_\xi$\footnote{Note that it is possible that $(X)_\xi$ is empty but then all the $(Y_\alpha)_\xi$'s will be empty from a certain point.}.
 Since $B$ witnesses that $U$ is basically generated, there is a function $f_0$ such that for every $f_0\leq_{bd}g$, $\bigcap_{\alpha<\kappa} A_{g(\al)}\in U$.

For each $\xi\in A$, consider the sequence $$Z^\xi_\alpha=\begin{cases} (Y_\alpha)_\xi & \xi\in A_\al\\
Z_*^\xi & \text{else}\end{cases}$$
where $Z_*^\xi$ is some (any) fixed element of $B_\xi$. Note that since $A_\alpha$ converges to $A$, then there is $\theta_\xi<\kappa$ such that for every $\alpha\in(\theta_\xi,\kappa)$, $\xi\in A_\al$ and therefore $Z^\xi_\alpha=(Y_\alpha)_\xi$.  
Also note that $A^\xi_\alpha$ converges to $(Y)_\xi$ and since $V_\xi$ is basically generated, there is some $f_\xi$ such that for every $f_\xi\leq_{bd} g$, $\cap_{\alpha<\kappa}A^\xi_{g(\alpha)}\in V_{\xi}$. 
Find a function $f^*$ such that for every $i\in\{0\}\cup A$, $f_i\leq_{bd}f^*$ and also $\mathrm{id}\leq_{bd}f^*$. 

Let $f^*\leq_{bd} g$. Since $f_0\leq_{bd}g$, $A_g:=\cap_{\alpha<\kappa} A_{g(\alpha)}\in U$ and $A_g\subseteq A$.
Fix $\xi\in A_g$.
We have that $f_\xi<_{bd}g$ and therefore $\cap_{\alpha<\kappa} Z^\xi_{g(\alpha)}\in V_\xi$. Since $\mathrm{id}\leq_{bd}g$, there is some $\alpha^*<\kappa$ such that  $\theta_\xi<g(\alpha)$ whenever $\alpha^*\leq\alpha$.
It follows that for every $\alpha^*\leq\alpha$, $Z^\xi_{g(\alpha)}=(Y_{g(\alpha)})_\xi$ and thus
$$(\cap_{\beta<\alpha^*}(Y_{g(\alpha)})_\xi)\cap(\cap_{\alpha<\kappa}Z^\xi_{g(\alpha)})\subseteq \cap_{\alpha<\kappa}(Y_{g(\alpha)})_\xi=:W_\xi$$
It follows that $W_\xi\in V_\xi$. Letting $X^*=\cup_{\xi\in A_g}\{\xi\}\times W_\xi$, we conclude that $X^*\in \sum_U V_\xi$ and also $X^*\subseteq \cap_{\alpha<\kappa}Y_{g(\alpha)}\subseteq \cap_{\alpha<\kappa}X_{g(\alpha)}$, concluding  that $\cap_{\alpha<\kappa}X_{g(\alpha)}\in \sum_U V_\xi$.

\end{proof}
\begin{remark}
    Note that if $\phi:\kappa\times\kappa\rightarrow\kappa$ is any bijection and $(A_i)_{i<\kappa}$ is a sequence of sets such $\lim_{\alpha\ra \kappa}A_\alpha=A$, then $\lim_{\alpha\ra\kappa}\phi''A_\alpha=\phi''A$, as there are many points $\mu<\kappa$ such that $\phi\restriction\mu\times\mu:\mu\times\mu\rightarrow \mu$. In particular, transferring a basically generated ultrafilter on $\kappa\times\kappa$ (such as in the previous theorem) to an ultrafilter on $\kappa$ will remain basically generated and thus non-Galvin.
\end{remark}

We point out that there are ultrafilters on $\om$ which are not basically generated and yet are not Tukey-top.
The first example  of this was seen in work of 
 Blass, Dobrinen and Raghavan \cite{Blass/Dobrinen/Raghavan15}.
 This particular ultrafilter was then shown by the second author   to have Tukey type which is the immediate successor of the Ramsey ultrafilter below it; further,   hierarchies of non-basically generated  ultrafilters whose cofinal types form downwards closed  chains of any finite length were also constructed
  \cite{DobrinenJSL15}. 
 Thus, it seems unlikely that basically generated ultrafilters  on measurable cardinals exactly capture non-Tukey topness.  However, no such examples are known at this time, so we ask: 

\begin{question}
Is it consistent that there are  ultrafilters over a measurable cardinal which are not  Tukey-top and also not 
 basically generated?
\end{question}

\section{Models of the Tukey-order}
In this section, we study the Tukey classes of some well-known models of measurable cardinals. We start with perhaps the best known such model, $L[U]$ which was studied extensively by K. Kunen in \cite{Kunen1970}.
\begin{theorem}\label{LofU}
In $L[U]$ there is no  ultrafilter among the $\sigma$-complete ultrafilters over $\kappa$ which is Tukey-top.
\end{theorem}
\begin{proof}
    In \cite{Parttwo}, Gitik and the first author observed that in $L[U]$ every $\kappa$ (or even $\sigma$) complete ultrafilter satisfies the Galvin property so in particular, by Theorem \ref{GlavinMainTheorem}, none of them are Tukey-top.
\end{proof}


  Recall that in Section 5, we showed   that for normal ultrafilters $ U , W  $ we have $ U \cdot W  \equiv_T U \times  W  $ (and also $ U ^n\equiv_T U $).
This can be used to say more about the model of $L[U]$: 

\begin{theorem}\label{analyse}
Let $U$ be a normal $\kappa$-complete ultrafilter over $\kappa$. In $L[U]$, the $\sigma$-complete ultrafilters over $\kappa$ form a single Tukey class which is the union of $\omega$-many Rudin Keisler equivalence classes.
\end{theorem}
\begin{proof}
Since $\kappa$ is the unique measurable in $L[U]$, every $\sigma$-complete ultrafilter is a $\kappa$-complete ultrafilter. It is well known that every $\kappa$-complete ultrafilter $W\in L[U]$ is RK isomorphic to $U^n$ for some $n<\omega$. In particular, by Theorem \ref{PowerCor}, $W\equiv_T U^n\equiv_T U$. Hence the Tukey equivalence class of $U$ is composed of the $\omega$-many RK equivalence classes with representatives $U^n$.
\end{proof}

\begin{remark}
In contrast, we point out that it is still unknown whether there is exactly one Tukey type for ultrafilters over $\om$ in 
    $L(\mathbb{R})[U]$, the Solovay model extended by a forced  Ramsey ultrafilter $U$ on $\om$.
\end{remark}

Although we do not have a Tukey-Top ultrafilter in $L[U]$, we do have a maximal class among the classes of $\kappa$-complete ultrafilters. 
As we will see in a moment, from larger cardinals we can obtain more complicated structures in the Tukey-order. 
In order to do that, let us establish a connection between the Mitchell order and the Tukey order.
\begin{definition}
    Let $U,W$ be $\sigma$-complete ultrafilters. We say that  $U$ is {\em Mitchell below} $W$, and denote it by $U\triangleleft W$, if $U\in M_W$, where $M_W$ is the (transitive collapse of the) ultrapower $V^\kappa/W$. 
\end{definition}
Similar to the Rudin-Keisler order, in the case of normal ultrafilters, the Tukey order is in some sense orthogonal to the Mitchell order:
\begin{theorem}\label{Mitchellorthogonal}
Suppose that $U\triangleleft W$ are two $\kappa$-complete ultrafilters over $\kappa$ and  $W$ is a $p$-point. Then $\neg( W\leq_{T}U)$.
\end{theorem}
\begin{proof}
Since $W$ is $p$-point, by Theorem \ref{Thm:DT20}, there is $f:[\kappa]^{<\kappa}\rightarrow[\kappa]^{<\kappa}$ a function which determines a continuous cofinal map from $U$ to $W$. Note that $f\in M_W$ (since $M_W$ is closed under $\kappa$-sequences and crit$(j_W)=\kappa$). Also since $U\triangleleft W$, $U\in M_W$ and therefore from $f$ and $U$ we can reproduce $W\in M_W$. This is known to be impossible (see for example \cite[Lemma 17.9.ii]{Jech2003}).
\end{proof}
\begin{corollary}
    Suppose that $U_1\triangleleft U_2\triangleleft\dots\triangleleft U_n$ are any normal ultrafilters, then $\neg (U_n\leq_T U_1\times\dots\times U_{n-1})$.
\end{corollary}
\begin{proof}
    In $M_{U_n}$ we have the ultrafilters $U_1,\dots,U_{n-1}$. If $U_n\leq_T U_1\times\dots\times U_{n-1}$, then this would be witnessed by a continuous map (Theorem \ref{Prodthm.20DT-kappa}) and the previous argument applies.
\end{proof}
\begin{corollary}
    Suppose that $o(\kappa)=\alpha$ for $\alpha\leq\omega$, then there is a Tukey-chain of $\kappa$-complete ultrafilters of order type $\alpha$.
\end{corollary}
\begin{proof}
    Indeed, let $\l U_n\mid n<\alpha\r$ be an increasing chain in the Mitchell order. Then $U_1\leq_{T}U_1\cdot U_2\le_T \dots$ is a $\leq_T$-chain. To see that this is strictly increasing, suppose toward a contradiction that $U_1\cdot {\dots} \cdot U_n\equiv_T U_1\cdot{\dots} \cdot U_{n+1}$.
Then $U_1\cdot{\dots}\cdot U_n\equiv_{T}U_1\times\dots\times U_n\geq_T U_{n+1}$, contradicting the previous corollary.  
\end{proof}
\begin{corollary}
   Let $L[\vec{U}]$ be the Mitchell model for $o(\kappa)=\omega$ and let $\l U_n\mid n<\omega\r$ be the $\triangleleft$-increasing sequence of ultrafilter on $\kappa$. Then the sequence $\l [U_1\times\dots\times U_n]_T\mid n<\omega\r$ is strictly increasing, cofinal and unbounded in the Tukey order. In particular, there is no maximal Tukey class.
\end{corollary}
In order to get a chain of order type $\omega+1$, we can start with the large cardinal assumption that $\kappa$ is a $\kappa$-compact cardinal (and in particular $o(\kappa)\geq\omega$). By Theorem \ref{GlavinMainTheorem}, there is a $\kappa$-complete ultrafilter $U$ such that $\neg \Gal(U,\kappa,2^\kappa)$ and such an ultrafilter is Tukey-top. Hence we can extend the chain from the previous corollary by one more ultrafilter. 

Actually, we can get such a chain from a much weaker assumption. Assume that the Rudin-Keisler order is $\sigma$-closed, i.e., given a 
Rudin-Keisler increasing
sequence 
$\l W_n\mid n<\omega\r$ 
of $\kappa$-complete ultrafilters on $\kappa$, there exists an ultrafilter $W$ such that for each $n$, $W_n\leq_{RK} W$. 
  \begin{proposition}
      If $o(\kappa)\geq\omega$ and the Rudin-Keisler order is $\sigma$-closed on $\kappa$ then there is a Tukey chain of length $\omega+1$.
  \end{proposition}
  \begin{proof}
      Similar to the previous construction, but the top ultrafilter will be any Rudin-Keisler bound for the sequence $\l U_1\cdot{\dots}\cdot U_n\mid n<\omega\r$. Since Rudin-Keisler reduction implies Tukey reduction, we are done. 
  \end{proof}
  The assumption that the Rudin-Keisler order is $\sigma$-closed follows from the \textit{$\omega$-gluing property}, due to Poveda and Hayut \cite[Lemma 4.1]{HayutPoveda}. In this paper, Poveda and Hayut proved that the $\omega$-gluing property is equiconsistent with $o(\kappa)=\omega_1$ which is much less than $\kappa$ being $\kappa$-compact.

  Finally, let us move back to the optimal assumption of a measurable cardinal. 
  Gitik and the first author \cite{OnPrikryandCohen} constructed a model with a $\kappa$-complete ultrafilter over $\kappa$ such that $\neg \Gal(U,\kappa,2^\kappa)$ from the minimal large cardinal assumption (i.e. a measurable cardinal). Thus in this model, we have a chain of length $2$. Finally, let us prove that using the usual Kunen-Paris construction  \cite{Kunen-Paris} of a model with many distinct normal ultrafilters we have a chain or ordertype $\omega+1$:
  \begin{theorem}\label{Kunen-Paris-Model}
      Assume \textsf{GCH} and that $\kappa$ is a measurable cardinal. Then there is a generic extension where \textsf{GCH} holds and there is a Tukey-chain of order type $\omega+1$.
  \end{theorem}
  \begin{proof}
      Consider the forcing from \cite{OnPrikryandCohen}, namely, let $\mathbb{P}=\mathbb{P}_\kappa*\mathbb{Q}_\kappa$ such that $\mathbb{P}_\kappa$ is just an Easton support iteration of $\Add(\alpha,\alpha^+)$ for inaccessible $\alpha$'s and trivial otherwise. Also, $\mathbb{Q}_\kappa$ is a $\mathbb{P}_\kappa$-name for $\Add(\kappa,\kappa^+)$. Let $G=G_\kappa*g$ be a $V$-generic filter for $\mathbb{P}$. Then by \cite[Theorem 2.6]{OnPrikryandCohen}, $2^\kappa=\kappa^+$ and in $V[G]$  there is a $\kappa$-complete ultrafilter $W$ over $\kappa$ such that $\neg \Gal(W,\kappa,2^\kappa)$. Over $V[G]$ let us force with $\Add(\kappa^{+},\kappa^{+3})$, let $H$ be $V[G]$-generic for this forcing. Then in $V[G,H]$,
      we note by the closure of $\Add(\kappa^+,\kappa^{+3})$, $W$ is still an ultrafilter in $V[G,H]$ and $\neg \Gal(W,\kappa,2^\kappa)$ holds. Also, the usual Kunen-Paris argument shows that in $V[G,H]$ there are $2^{(2^\kappa)}=2^{\kappa^{+}}=\kappa^{+3}$-many normal ultrafilters over $\kappa$.  Now let us apply Corollary \ref{ManyP-pointsCor} to conclude that there are $(2^\kappa)^+$-many non-normal ultrafilters which are Tukey incomparable. Take any $\omega$-many of those ultrafilters; the same chain of products as before produces a Tukey strictly increasing chain and the top one which is the non-Galvin one.
  \end{proof}
   Kuratowski and Sierpi\'{n}ski \cite{Kuratowski} (See also \cite[Thm. 46.1]{CombinatorialST}) proved the following:
   \begin{lemma}\label{Kur}
       For any cardinal $\lambda$ and $k<\omega$, given any function $f:[\lambda^{+k}]^k\rightarrow [\lambda^{+k}]^{<\lambda}$ such that for any $x_1,\dots,x_k$, $\{x_1,\dots,x_k\}\cap f(x_1,\dots,x_k)=\emptyset$, there is a free set $Y$ of size $k+1$, namely, there are distinct $y_1,\dots,y_{k+1}\in \lambda^{+k}$ such that for every $x_1,\dots,x_k\in[\{y_1,\dots,y_{k+1}\}]^{k}$, $f(x_1,\dots,x_k)\cap \{y_1,\dots,y_{k+1}\}=\emptyset$.
   \end{lemma}
  \begin{corollary}
     For every $0<k<\omega$ it is consistent that  $(P(k)\setminus\{\emptyset\},\subseteq)$ can be embedded
     into the Tukey classes.
     \end{corollary}
      \begin{proof}
            Assume $k>0$. Consider the Kunen-Paris model $M$ where $2^{2^\kappa}=\kappa^{+k+2}$ and $2^\kappa=\kappa^+$. In particular, the set of normal ultrafilters $\beta_{\text{nor}}(\kappa)$ has size $\kappa^{+k+2}$ in $M$. Consider the function:
            $F:[\beta_{\text{nor}}(\kappa)]^k\rightarrow P(\beta_{\text{nor}}(\kappa))$ defined by
            $$(*) \ \ F(U_1,\dots,U_k)=\{W\in \beta_{\text{nor}}(\kappa)\mid W\leq_T U_1\times\dots\times U_k\}\setminus \{U_1,\dots,U_k\}$$
            By Theorem \ref{Prodthm.20DT-kappa} $|F(U_1,\dots,U_k)|\leq 2^\kappa=\kappa^+$.    
            Apply lemma \ref{Kur} for $\lambda=(2^\kappa)^+=\kappa^{+2}$ to find a free set of ultrafilters $\{U_1,\dots,U_{k+1}\}$ for the function $F$ of size $k+1$. It follows that for each $1\leq i\leq k+1$, $$\neg( U_i\leq_T U_1\times\dots\times U_{i-1}\times U_{i+1}\times\dots\times U_{k+1}).$$ 
            It follows that if $a\subsetneq b\subseteq \{1,\dots,k+1\}$, then $\prod_{i\in a}U_i<_T\prod_{j\in b}U_j$. Just otherwise, take any $j\in b\setminus a$, then $U_j\leq_T \prod_{i\in a}U_i\leq_T \prod_{i\neq j}U_i$ contradicting $(*)$. For the same reason, it follows that if $a\neq b$ then $\prod_{i\in a}U_i\not\equiv_T\prod_{j\in b}U_j$. We conclude that the map $g$ defined on 
        $P(\{1,\dots,k+1\})$ 
         by $g(a)=[\prod_{i\in a}U_i]_T$ is an order embedding.
     \end{proof}

\begin{remark}
    We point out some similar results on $\om$ and the contrast between embedding and exact Tukey structures.
  Dobrinen, Mijares, and Trujillo \cite{Dobrinen/Mijares/Trujillo14} showed that it is consistent with ZFC that $([\om]^{<\om},\sse)$ appears as an {\em initial} (downwards closed in the Tukey order) structure  among the cofinal types of  p-points on $\om$.
Raghavan and Shelah \cite{Raghavan/Shelah17} later showed that it is consistent that $\mathcal{P}(\om)/\mathrm{fin}$ embeds into the Tukey types of p-points.
Except for Theorem \ref{analyse}, the above results in this section so far are of this latter sort, showing that certain partial orders embed into the Tukey types of a given structure.
The Ultrafilter Axiom  and the minimality conjecture below will 
ensure 
more examples where the Tukey structure is {\em exactly} known. 
\end{remark}

By a result of Todorcevic 
in \cite{Raghavan/Todorcevic12}, every Ramsey ultrafilter on $\omega$ is 
Tukey minimal. The proof  utilizes a  theorem of Pudl\'{a}k and R\"{o}dl 
\cite{Pudlak/Rodl82} canonizing equivalence relations on barriers;
this theorem 
 does not generalize to uncountable cardinals.  Hence a crucial open problem is the following:
\begin{question}\label{normalminimalQuestion}
    Are normal ultrafilters on a measurable cardinal $\kappa$ Tukey minimal among $\kappa$ complete ultrafilters?
\end{question}
We conjecture a positive answer for this question, and call this the \textit{minimality conjecture}. 
The \textit{general minimality conjecture} is the assertion that if $U_1,..,U_n,W$ are normal ultrafilters over $\kappa$ and $W\notin\{U_1,\dots,U_n\}$ then $\neg (W\leq_T \prod_{i=1}^n U_i)$

There is  good reason to believe that the general minimality conjecture holds. Indeed, it is well known that normal ultrafilters are $RK$-minimal and the only ultrafilters which are $RK$ below a product of normals are the ultrafilters participating in the product. Also, by Theorem \ref{Prodthm.20DT-kappa}, if $f:\prod_{i=1}^nU_i\rightarrow W$ witness that $U_1\cdot{\dots}\cdot U_n\geq_TW$ then $f$ is actually determined by a function from  $([\kappa]^{<\kappa})^n$ to $[\kappa]^{<\kappa}$.
On  $\om$  it is known how to make such maps 
into Rudin-Keisler projections under various hypotheses on the $U_i$ and $W$
 (see \cite{Raghavan/Todorcevic12} and \cite{DobrinenFund20}).
\begin{example}
    As we have seen above, if $U_0\triangleleft U_1$ then $\neg( U_1\triangleleft U_0)$. A natural  approach for trying to  produce a counter-example for the minimality conjecture is to try to prove that $U_0\triangleleft U_1$ implies $U_0\leq_T U_1$ (where $U_0,U_1$ are assumed to be normal). There is a natural map for this. Since $U_0\in M_{U_1}$  there is a function $f:\kappa\rightarrow V_\kappa$ such that $[f]_{U_1}=U_0$. Recall that $U_1$ is normal and therefore $\kappa=[id]_{U_1}$. By Lo\'{s} Theorem, since $$M_{U_1}\models [f]_{U_1}\text{ is a normal ultrafilter over }[id]_{U_1},$$ we have that $E=\{\alpha<\kappa\mid f(\alpha)\text{ is a normal ultrafilter over }\alpha\}\in U_1$.  For every $\alpha\in E$, denote $f(\alpha)=U_{0,\alpha}$. 
    
Let us define a function $f:U_0\rightarrow U_1$ by
$$f(X)=\{\alpha<\kappa\mid X\cap\alpha\in U_{0,\alpha}\}$$
This is well defined since for every $X\subseteq \kappa$, $X\in U_0$ iff $\{\alpha<\kappa\mid X\cap\alpha\in U_{0,\alpha}\}\in U_1$.
We claim that this $f$ is not unbounded.
Indeed, let $A_\alpha=\kappa\setminus\{\alpha\}$, then $\{A_\alpha\mid \alpha<\kappa\}$ is unbounded in $U_0$ since $\cap_{\alpha<\kappa}A_\alpha=\emptyset$. On the other hand, for every $\alpha$, $f(A_\alpha)=\{\alpha<\kappa\mid A_\alpha\cap \alpha\in U_{0,\alpha}\}=E$ and therefore $f''\{A_\alpha\mid \alpha<\kappa\}$ is bounded. 
\end{example}

The minimality conjecture can be used to completely characterize the Tukey order in the canonical inner models for relatively small cardinals (and possibly further). For example, we can classify the Tukey classes in the Mitchell models $L[\vec{U}]$ where $\vec{U}$ is a coherent sequence of normal measures \cite{MitchelModel} of length which is not too long (below $\kappa^{++}$). For this we will need the following easy corollary from the general minimality conjecture. 
\begin{corollary}\label{easycor} Let $U_1,U_2,\dots,U_n$ be distinct normal ultrafilters.  Let $1\leq i\leq n$ and $\emptyset\neq a\subseteq\{1,\dots,n\}$.
    \begin{enumerate}
        \item If $i\in a$ then $U_i\leq_{RK} \prod_{j\in a}U_j$.
        \item If the general minimality conjecture holds then for $i\notin a$,  $U_i$ and $\prod_{j\in a}U_j$ are Tukey incomparable.
    \end{enumerate}
\end{corollary}
\begin{proof}
    $(1)$ is clear. For $(2)$, we note that by the minimality assumption $\neg (U_i\leq_T\prod_{j\in a}U_j)$ and also $\neg(U_i\geq_T\prod_{j\in a}U_j)$. Otherwise since $a\neq\emptyset$, there is $j_0\in a$ and $j_0\neq i$, we get that $U_i\geq_T \prod_{j\in a}U_j\geq_T U_{j_0}$. But this  contradicts the minimality conjecture since $U_i\neq U_{j_0}$.
\end{proof}

Next, let us state some known results which generalize Kunen's classification of ultrafilters in $L[U]$. The framework in which we are going to analyze the structure of the Tukey classes is due to G. Goldberg \cite{GoldbergUA}, who discovered the \textit{Ultrapower Axiom}, which holds in all known canonical inner models, and in particular in the models $L[\vec{U}]$. Goldberg showed that in the context of \textsf{UA}, the structure of $\kappa$-complete ultrafilters is quite rigid and well understood. The first result which is relevant for our needs is the generalization of Kunen's classification of ultrafilters in the model $L[U]$. This is a direct corollary of \cite[Lemma 5.3.20]{GoldbergUA}: 

\begin{lemma}[Goldberg]\label{KunenGen}
Assume \textsf{UA} and that $\kappa$ is a measurable cardinal with $o(\kappa)<2^{(2^\kappa)}$. Then every $\kappa$-complete ultrafilter is Rudin-Keisler equivalent to a finite product of normal ultrafilters.
\end{lemma}
The following remarkable theorem of Goldberg  is also relevant for us:
\begin{theorem}[Goldberg]\label{MitchellIsLin}
    Under \textsf{UA}, the Mitchell order is linear (when restricted to normal measures).
\end{theorem}
Let us prove that the above results suffice to completely characterize the Tukey classes under  \textsf{UA} and small enough large cardinal assumption:
\begin{theorem}\label{UACor}
    Assume \textsf{UA} and the general minimality conjecture. Suppose that $\kappa$ is a measurable cardinal with $o(\kappa)<2^{(2^\kappa)}$. Then the Tukey classes of ultrafilters is isomorphic to $([o(\kappa)]^{<\omega},\subseteq)$.
\end{theorem}
\begin{proof}
    By theorem \ref{MitchellIsLin} let $\l U_\alpha\mid \alpha<\gamma\r$ be the $\triangleleft$-increasing sequence of all the normal ultrafilters, where $\gamma=o(\kappa)$. Next, by Lemma \ref{KunenGen}, any $\kappa$-compete ultrafilter is Rudin-Keisler isomorphic to a finite product of normal ultrafilters which are listed in the sequence $\l U_\alpha\mid \alpha<\gamma\r$. As in Theorem \ref{analyse}, we have that for any ultrafilter $W$ on $\kappa$, $[W]_T=[\prod_{i\in a}U_i]_T$ for some finite set $a\subseteq \gamma$. Let us check that the function $f:[\gamma]^{<\kappa}\rightarrow 
    \{[U]_T:U\in\mathcal{U}\}$,
    where $\mathcal{U}$ is the set of $\kappa$-complete  ultrafilters on $\kappa$, defined by $f(a)=[\prod_{i\in a}U_i]_T$ is an order isomorphism.
    Clearly if $a\subseteq b$ then $\prod_{i\in a}U_i\leq_{RK} \prod_{i\in b}U_i$. If $a\not\subseteq b$ then we cannot have $\prod_{i\in a}U_i\leq_T\prod_{i\in b}U_i$ since this will contradict Corollary \ref{easycor}. Also if $a\subsetneq b$, then again Corollary \ref{easycor} implies that $\prod_{i\in a} U_i<_T \prod_{i\in b}U_i$.
\end{proof}

Since the models $L[\vec{U}]$  are models of \textsf{UA}, we obtain the following corollary: 
\begin{corollary}\label{mainClass}
Assume the general minimality conjecture and that $ \vec{U} $ is a coherent sequence for $o(\kappa)=\gamma<2^{2^{\kappa}}$. Then in $L[\vec{U}]$, the Tukey order on equivalence classes of $\kappa$-complete ultrafilters is isomorphic to $([\gamma]^{<\omega},\subseteq)$.
\end{corollary}
\begin{example}
Suppose that $\vec{U}$ is a coherent sequence for $o(\kappa)=2$. Then in $L[\vec{U}]$ there are $2$ minimal Tukey classes for $\kappa$ corresponding to $U(\kappa,0),U(\kappa,1)$, and there is a third class corresponding to $U(\kappa,0)\cdot U(\kappa,1)$ strictly above $U(\kappa,0),U(\kappa,1)$ (and that's it!) 
\end{example}

\begin{remark}
    
In the case of $(U,\supseteq^*)$, in the models $L[\vec{U}]$ there is a single class of ultrafilters since any two normal ultrafilters are equivalent (by \ref{p-point}) and therefore products are are also equivalent.

\end{remark}

\section{Open problems and further directions}
 
 In an attempt to obtain a model with longer Tukey-chains, Gitik's construction from \cite{Gitik1986}  turning a Mitchell increasing sequence of ultrafilters into a Rudin-Keisler increasing sequence seems like a promising direction to generate longer chains of non-Tukey equivalent ultrafilters:
 \begin{question}
 In the model of \cite{Gitik1986}, is there a long Tukey-chain?
 \end{question}
 Note that the extensions of the ultrafilters are not going to be $q$-point nor $p$-point and therefore our results about such ultrafilters are not applicable. 

 Todorcevic classified the ultrafilters which are Tukey below a Ramsey ultrafilter $U$ on $\om$ to
consist exactly of 
 the Rudin-Keisler classes of countable Fubini iterates of $U$ (Theorem 24, \cite{Raghavan/Todorcevic12}).
 As all countable Fubini iterates of a Ramsey ultrafilter are Tukey equivalent,
it  follows that  Ramsey ultrafilters on $\om$ are Tukey minimal.
 We do not know whether any of these statements  generalize to  measurable cardinals:
 \begin{question}
     Suppose that $U$ is normal. Is it true that whenever $V\leq_T U$ then $V\equiv_{RK} U^n$ for some $n<\omega$?
 \end{question}
 This is certainly the case under \textsf{UA}, but is it true in general?
Note also that the minimality conjecture from Question \ref{normalminimalQuestion}.
The problem with trying to directly extend Theorem 24 in \cite{Raghavan/Todorcevic12} to normal ultrafilters, is that Todorcevic's proof heavily relies on  
combinatorial features of 
fronts and barriers on $\om$ such as the Nash-Williams \cite{NashWilliams65} and Pudl\'{a}k-R\"{o}dl 
\cite{Pudlak/Rodl82} Theorems. These Theorems are not available on measurable cardinals, as their relevant generalizations would in particular imply the existence of large homogeneous sets for colorings $f:[\kappa]^\omega\rightarrow 2$. This is known to be impossible by the classical example of Erd\H{o}s and Rado \cite{Erds1952CombinatorialTO} (see also \cite[Proposition 7.1]{kanamori1994}) 

 In the Kunen-Paris model we used the Kuratowski-Sierpi\'{n}ski  characterization of $\lambda^{+k}$. It is known that this is best possible if we consider any possible map. 
 However, we apply this characterization to the specific map $$F(U_1,\dots,U_k)=\{W\mid W\leq_T U_1\times\dots\times U_k\}\setminus\{U_1,\dots,U_k\}.$$
 This map has further properties, for example, if $x_1,\dots,x_k\in F(U_1,\dots,U_k)$ then $F(x_1,\dots,x_k)\subseteq F(U_1,\dots,U_k)$. 
 \begin{question}
     How large of a free set can we guarantee for the function $F$? Can we guarantee an infinite set? What about the case where $F:[\beta_{\text{nor}}(\kappa)]^{<\omega}\rightarrow P(\beta_{\text{nor}}(\kappa))$?
 \end{question}
 \subsection*{Acknowledgment} We would like to thank the participants of the TAU Set Theory Seminar, and especially M. Gitik and M. Magidor for their valuable comments during the presentation of this work during July 2023. Also, we would like to express our gratitude to the anonymous referee of this paper for carefully going over the paper, and for their mathematical and non-mathematical comments.

\bibliographystyle{amsplain}
\bibliography{ref}
\end{document}